\documentclass[10pt, reqno]{amsart}

\def\BibTeX{{\rm B\kern-.05em{\sc i\kern-.025em b}\kern-.08em
    T\kern-.1667em\lower.7ex\hbox{E}\kern-.125emX}}

\newtheorem{thm}{Theorem}[section]

\newtheorem{lem}[thm]{Lemma}

\theoremstyle{definition}

\theoremstyle{remark}
\newtheorem{rem}{Remark}[section]

\numberwithin{equation}{section}

    \newcommand{\floor}[1]{\lfloor#1\rfloor}

    \newcommand{\EE}{\mathbb{E}}

    \renewcommand{\Pr}{\operatorname{P}}

    \newcommand{\dto}{\xrightarrow{d}}
    
    \newcommand{\vto}{\xrightarrow{v}}

    \newcommand{\stas}{\stackrel{\rm a.s.}{\longrightarrow}}

    \newcommand{\rmd}{\mathrm{d}}

\newcommand{\PP}{\mathrm{P}}

\newcommand{\be}{\begin{equation}}
    \newcommand{\ee}{\end{equation}}

\begin{document}

\title[Joint convergence of partial sum and maxima for linear processes] 
{Joint functional convergence of partial sum and maxima for linear processes}

%
\author{Danijel Krizmani\'{c}}

\address{Danijel Krizmani\'{c}\\ Department of Mathematics\\
        University of Rijeka\\
        Radmile Matej\v{c}i\'{c} 2, 51000 Rijeka\\
        Croatia}
\email{dkrizmanic@math.uniri.hr}





\begin{abstract}
For linear processes with independent identically distributed innovations that are regularly varying with tail index $\alpha \in (0, 2)$, we study functional convergence of the joint partial sum and partial maxima processes.
We derive a functional limit theorem under certain assumptions on the coefficients of the linear processes which enable the functional convergence to hold in the space of $\mathbb{R}^{2}$--valued c\`{a}dl\`{a}g functions
on $[0, 1]$ with the Skorohod weak $M_{2}$ topology. Also a joint convergence in the $M_{2}$ topology on the first coordinate and in the $M_{1}$ topology on the second coordinate is obtained.
\end{abstract}

\maketitle

\section{Introduction}
\label{intro}

It is known that the joint partial sum and partial maxima processes constructed from i.i.d.~regularly varying random variables
with the tail index $\alpha \in (0,2)$
satisfy the functional limit theorem
with $(V(\,\cdot\,), W(\,\cdot\,))$ as a limit, where $V(\,\cdot\,)$ is a stable L\'{e}vy process and $W(\,\cdot\,)$ an extremal process, see Chow and Teugels~\cite{ChTe79} and Resnick~\cite{Re86}. The convergence takes place in the space $D([0, 1], \mathbb{R}^{2})$ of $\mathbb{R}^{2}$--valued c\`adl\`ag functions on $[0, 1]$ with the Skorohod $J_{1}$ topology.

In this paper we study functional convergence of a special class of weakly dependent random variables, the linear processes or moving averages processes. Due to possible clustering of large values, functional convergence fails to hold with respect to the $J_{1}$ topology, and hence we will have to use a somewhat weaker topology, namely the Skorohod weak $M_{2}$ topology. In the proofs of our results we will use the methods and results which appear in Basrak and Krizmani\'{c}~\cite{BaKr14}, where they obtained functional convergence of partial sum processes with respect to Skorohod (standard or strong) $M_{2}$ topology.

We proceed by stating the problem precisely.
Let $(Z_{i})_{i \in \mathbb{Z}}$ be an i.i.d.~sequence of regularly varying random variables with  index of regular variation $\alpha \in (0,2)$.
In particular, this means that
$$ \Pr(|Z_{i}| > x) = x^{-\alpha} L(x), \qquad x>0,$$
where $L$ is a slowly varying function at $\infty$.
Let $(a_{n})$ be a sequence of positive real numbers such that
\be\label{eq:niz}
n \Pr (|Z_{1}|>a_{n}) \to 1,
\ee
as $n \to \infty$. Then $a_{n} \to \infty$. Regular
variation of $Z_{i}$ can be expressed in terms of
vague convergence of measures on $\EE = \overline{\mathbb{R}} \setminus \{0\}$:
\begin{equation}
  \label{eq:onedimregvar}
  n \Pr( a_n^{-1} Z_i \in \cdot \, ) \vto \mu( \, \cdot \,) \qquad  \textrm{as} \ n \to \infty,
\end{equation}
with the measure $\mu$ on $\EE$  given by
\begin{equation}
\label{eq:mu}
  \mu(\rmd x) = \bigl( p \, 1_{(0, \infty)}(x) + r \, 1_{(-\infty, 0)}(x) \bigr) \, \alpha |x|^{-\alpha-1} \, \rmd x,
\end{equation}
where
\be\label{eq:pq}
p =   \lim_{x \to \infty} \frac{\Pr(Z_i > x)}{\Pr(|Z_i| > x)} \qquad \textrm{and} \qquad
  r =   \lim_{x \to \infty} \frac{\Pr(Z_i \leq -x)}{\Pr(|Z_i| > x)}.
\ee
When $\alpha \in (1,2)$ it holds that $\mathrm{E}(Z_{1}) < \infty$.
We study
the moving average process of the form
 $$X_{i} = \sum_{j=-\infty}^{\infty}\varphi_{j}Z_{i-j}, \qquad i \in \mathbb{Z},$$
where the constants $\varphi_{j}$ are such that the above series is a.s.~convergent. One sufficient condition for that is
\begin{equation}\label{e:fddconv}
 \sum_{j=-\infty}^{\infty}|\varphi_{j}|^{\delta} < \infty \quad \textrm{for some} \ 0 < \delta < \alpha,\,\delta \leq 1
\end{equation}
(see Theorem 2.1 in Cline~\cite{Cl83} or Resnick~\cite{Re87}, Section 4.5). As noted in~\cite{BaJaLo16}, condition (\ref{e:fddconv}) excludes some important cases, like the case of strictly $\alpha$-stable random variables $(Z_{i})$ with $\sum_{j}|\varphi_{j}|^{\alpha} < \infty$, but $\sum_{j}|\varphi_{j}|^{\delta}=\infty$ for every $\delta < \alpha$. To resolve this issue some new conditions, weaker then (\ref{e:fddconv}) for $\alpha \leq 1$, were proposed by Balan et al.~\cite{BaJaLo16}, Corollaries 4.6 and 4.9.
 In~\cite{As83} it was observed that if additionally holds
$$ \begin{array}{rl}
                                   \mathrm{E}(Z_{1})=0, & \quad \textrm{if} \ \alpha \in (1,2),\\[0.4em]
                                   Z_{1} \ \textrm{is symmetric}, & \quad \textrm{if} \ \alpha =1,
                                 \end{array}$$
then the series defining $X_{i}$ is a.s.~convergent if, and only if,
\begin{equation}\label{e:fddconv2}
 \sum_{j =-\infty}^{\infty}|\varphi_{j}|^{\alpha} L(|\varphi_{j}|^{-1}) < \infty
\end{equation}
(see also Proposition 5.4 in~\cite{BaJaLo16}). Note that condition (\ref{e:fddconv}) implies $\sum_{i=-\infty}^{\infty}|\varphi_{i}| < \infty$. The same holds if condition (\ref{e:fddconv2}) is satisfied when $\alpha \in (0,1)$.

Our goal is to find sufficient conditions such that, with respect to some Skorohod topology on $D([0,1], \mathbb{R}^{2})$,
\be\label{eq:AvTaintr}
\Big( \sum_{i=1}^{\floor{n\,\cdot}} \frac{X_{i} - b_{n}}{a_{n}},  \bigvee_{i=1}^{\floor{n\,\cdot}}\frac{X_{i}}{a_{n}} \Big) \dto  (\beta V(\,\cdot\,), W(\,\cdot\,)),
\ee
in $D([0,1], \mathbb{R}^{2})$, where $V(\,\cdot\,)$ is an $\alpha$--stable L\'{e}vy process, $W(\,\cdot\,)$ is an extremal process, $b_{n}$ are appropriate centering constants and $ \beta =  \sum_{j=-\infty}^{\infty}\varphi_{j} \neq 0$. $D([0,1], \mathbb{R}^{2})$ denotes the space of right continuous $\mathbb{R}^{2}$--valued functions on $[0,1]$ with left limits.

Recall here some basic facts on L\'{e}vy processes and extremal processes. The distribution of a L\'{e}vy process $V (\,\cdot\,)$ is characterized by its
characteristic triple (i.e. the characteristic triple of the infinitely divisible distribution
of $V(1)$). The characteristic function of $V(1)$ and the characteristic triple
$(a, \nu', b)$ are related in the following way:
 $$
  \mathrm{E} [e^{izV(1)}] = \exp \biggl( -\frac{1}{2}az^{2} + ibz + \int_{\mathbb{R}} \bigl( e^{izx}-1-izx 1_{[-1,1]}(x) \bigr)\,\nu'(\rmd x) \biggr)$$
for $z \in \mathbb{R}$, where $a \ge 0$, $b \in \mathbb{R}$ are constants, and $\nu'$ is a measure on $\mathbb{R}$ satisfying
$$ \nu' ( \{0\})=0 \qquad \text{and} \qquad \int_{\mathbb{R}}(|x|^{2} \wedge 1)\,\nu'(\rmd x) < \infty.$$
 We refer to Sato~\cite{Sa99} for a textbook treatment of
L\'{e}vy processes.
The distribution of an nonnegative extremal process $W(\,\cdot\,)$ is characterized by its exponent measure $\nu''$ in the following way:
$$ \PP (W(t) \leq x ) = e^{-t \nu''(x,\infty)}$$
for $t>0$ and $x>0$, where $\nu''$ is a measure on $(0,\infty)$ satisfying
$ \nu'' (\delta, \infty) < \infty$
for any $\delta >0$ (see Resnick~\cite{Resnick07}, page 161).

  If $X_{i}$ is a finite order moving average with at least two nonzero coefficients, then the convergence in (\ref{eq:AvTaintr}) cannot hold in the $J_{1}$ sense, since as showed by Avram and Taqqu~\cite{AvTa92} the $J_{1}$ convergence fails to hold for the first components of the processes in (\ref{eq:AvTaintr}), i.e. for partial sum processes. Astrauskas~\cite{As83} and Davis and Resnick~\cite{DaRe85} showed that the normalized sums of $X_i$'s under~(\ref{e:fddconv}) converge in distribution to a stable random variable. Basrak and Krizmani\'{c}~\cite{BaKr14} replaced this convergence by weak convergence with respect to the Skorohod $M_{_2}$ topology, i.e. they showed that the convergence for partial sums,
$$  \sum_{i=1}^{\floor{n\,\cdot}} \frac{X_{i} - b_{n}}{a_{n}} \dto  \beta V(\,\cdot\,)$$
holds in the $M_{2}$ topology with the following assumption on the coefficients $\varphi_{i}$:
$\varphi_{j}=0$ for $j < 0$, $\varphi_{0}, \varphi_{1}, \ldots \in \mathbb{R}$ and for every $s=0,1,2,\ldots$
\begin{equation}\label{e:BaKrassumption}
 0 \leq \sum_{j=0}^{s}\varphi_{j} \bigg/ \sum_{j=0}^{\infty}\varphi_{j} \leq 1.
 \end{equation}
The characteristic
 triple of the limiting process $V(\,\cdot\,)$  is of the form $(0,\mu,b)$, with $\mu$ as in $(\ref{eq:mu})$ and
$$ b = \left\{ \begin{array}{cc}
                                   0, & \quad \alpha = 1\\[0.4em]
                                   (p-r)\frac{\alpha}{1-\alpha}, & \quad \alpha \in (0,1) \cup (1,2)
                                 \end{array}\right..$$
 As for the partial maxima, Resnick~\cite{Re87} showed that if $\varphi_{+}p + \varphi_{-}r>0$, then, as $n \to \infty$,
$$  \bigvee_{i=1}^{\floor{n\,\cdot}} \frac{X_{i}}{a_{n}} \dto   W(\,\cdot\,)$$
in the $J_{1}$ topology, where
$$\varphi_{+}=\max \{ \varphi_{j} \vee 0 : j \in \mathbb{Z}\}, \qquad \varphi_{-}= \max \{ -\varphi_{j} \vee 0 : j \in \mathbb{Z}\},$$
and $W(\,\cdot\,)$ is an extremal process with exponent measure
$$ \nu(dx) = (\varphi_{+}^{\alpha}p + \varphi_{-}^{\alpha}r) \alpha x^{-\alpha-1}1_{(0,\infty)}(x)\,dx.$$
(see Proposition 4.28 in Resnick~\cite{Re87}).


In this article we will show that, under assumptions (\ref{e:BaKrassumption}) and $\varphi_{+}p + \varphi_{-}r>0$,
 relation (\ref{eq:AvTaintr}) holds in  the weak $M_{2}$ topology.
 In order to do so, we first in Section~\ref{S:M2top} recall the precise definition of the weak $M_{2}$ topology, and then in Section~\ref{S:MA} we proceed by proving (\ref{eq:AvTaintr}) for finite order moving average processes and then we
extend this to infinite order moving average processes. At the end in Remark~\ref{r:jcM2M1} we discuss joint convergence in (\ref{eq:AvTaintr}) in the $M_{2}$ topology on the first coordinate and in the $M_{1}$ topology on the second coordinate.

\section{Skorohod $M_{2}$ topologies}\label{S:M2top}

We start with a definition of the Skorohod weak $M_{2}$ in a general space $D([0,1], \mathbb{R}^{d})$ of $\mathbb{R}^{d}$--valued c\`{a}dl\`{a}g functions on
$[0,1]$.

The weak $M_{2}$ topology on $D([0, 1], \mathbb{R}^{d})$ is defined using completed graphs.
For $x \in D([0,1],
\mathbb{R}^{d})$ the completed (thick) graph of $x$ is the set
\[
  G_{x}
  = \{ (t,z) \in [0,1] \times \mathbb{R}^{d} : z \in [[x(t-), x(t)]]\},
\]
where $x(t-)$ is the left limit of $x$ at $t$ and $[[a,b]]$ is the product segment, i.e.
$[[a,b]]=[a_{1},b_{1}] \times [a_{2},b_{2}] \ldots \times [a_{d},b_{d}]$
for $a=(a_{1}, a_{2}, \ldots, a_{d}), b=(b_{1}, b_{2}, \ldots, b_{d}) \in
\mathbb{R}^{d}$. We define an
order on the graph $G_{x}$ by saying that $(t_{1},z_{1}) \le
(t_{2},z_{2})$ if either (i) $t_{1} < t_{2}$ or (ii) $t_{1} = t_{2}$
and $|x_{j}(t_{1}-) - z_{1j}| \le |x_{j}(t_{2}-) - z_{2j}|$
for all $j=1, 2, \ldots, d$. The relation $\le$ induces only a partial
order on the graph $G_{x}$. A weak $M_{2}$ parametric representation
of the graph $G_{x}$ is a continuous function $(r,u)$
mapping $[0,1]$ into $G_{x}$, such that $r$ is nondecreasing with $r(0)=0$, $r(1)=1$ and $u(1)=x(1)$ ($r$ is the
time component and $u$ the spatial component). Let $\Pi_{w}(x)$ denote the set of weak $M_{2}$
parametric representations of the graph $G_{x}$. For $x_{1},x_{2}
\in D([0,1], \mathbb{R}^{d})$ define
\[
  d_{w}(x_{1},x_{2})
  = \inf \{ \|r_{1}-r_{2}\|_{[0,1]} \vee \|u_{1}-u_{2}\|_{[0,1]} : (r_{i},u_{i}) \in \Pi_{w}(x_{i}), i=1,2 \},
\]
where $\|x\|_{[0,1]} = \sup \{ \|x(t)\| : t \in [0,1] \}$. Now we
say that $x_{n} \to x$ in $D([0,1], \mathbb{R}^{d})$ for a sequence
$(x_{n})$ in the weak Skorohod $M_{2}$ (or shortly $WM_{2}$)
topology if $d_{w}(x_{n},x)\to 0$ as $n \to \infty$.

If we replace above the graph $G_{x}$ with the completed (thin) graph
\[
  \Gamma_{x}
  = \{ (t,z) \in [0,1] \times \mathbb{R}^{d} : z= \lambda x(t-) + (1-\lambda)x(t) \ \text{for some}\ \lambda \in [0,1] \},
\]
and a weak $M_{2}$ parametric representation with a strong $M_{2}$ parametric representation (i.e. a continuous function $(r,u)$ mapping $[0,1]$ onto $\Gamma_{x}$ such that $r$ is nondecreasing), then we obtain the standard (or strong) $M_{2}$ topology. This topology is stronger than the weak $M_{2}$ topology, but they coincide if $d=1$. Both topologies are weaker than the more frequently used Skorohod $J_{1}$ and $M_{1}$ topologies.
The $M_{2}$ topology on $D([0,1], \mathbb{R})$ can be generated using the Hausdorff metric on the spaces of graphs. For $x_{1},x_{2} \in D([0,1], \mathbb{R})$ define
$$ d_{M_{2}}(x_{1}, x_{2}) = \bigg(\sup_{a \in \Gamma_{x_{1}}} \inf_{b \in \Gamma_{x_{2}}} d(a,b) \bigg) \vee \bigg(\sup_{a \in \Gamma_{x_{2}}} \inf_{b \in \Gamma_{x_{1}}} d(a,b) \bigg),$$
where $d$ is the metric on $\mathbb{R}^{2}$ defined by $d(a,b)=|a_{1}-b_{1}| \vee |a_{2}-b_{2}|$ for $a=(a_{1},a_{2}), b=(b_{1},b_{2}) \in \mathbb{R}^{2}$.

The weak $M_{2}$ topology on $D([0,1], \mathbb{R}^{2})$
coincides with the (product) topology induced by the metric
\begin{equation}\label{e:defdp}
 d_{p}(x_{1},x_{2})= \max_{j=1,2}d_{M_{2}}(x_{1j},x_{2j})
\end{equation}
 for $x_{i}=(x_{i1}, x_{i2}) \in D([0,1],
 \mathbb{R}^{2})$, $i=1,2$. For detailed discussion of the strong and weak $M_{2}$ topologies we refer to
Whitt~\cite{Whitt02}, sections 12.10--12.11.

In the next section we will use the following lemma.

\begin{lem}\label{l:weakM2transf}
Let $(A_{n}, B_{n})$, $n=0,1,2,\ldots$, be stochastic processes in $D([0,1], \mathbb{R}^{2})$ such that, as $n \to \infty$,
\begin{equation}\label{e:lemM2}
(A_{n}(\,\cdot\,), B_{n}(\,\cdot\,)) \dto (A_{0}(\,\cdot\,), B_{0}(\,\cdot\,))
\end{equation}
in $D([0,1], \mathbb{R}^{2})$ with the weak $M_{2}$ topology. Suppose $x_{n}$, $n=0,1,2,\ldots$, are elements of $D([0,1], \mathbb{R})$ with $x_{0}$ being continuous, such that, as $n \to \infty$,
$$x_{n}(t) \to x_{0}(t) $$
uniformly in $t$. Then
$$ (A_{n}(\,\cdot\,) +x_{n}(\,\cdot\,), B_{n}(\,\cdot\,)) \dto (A_{0}(\,\cdot\,) + x_{0}(\,\cdot\,), B_{0}(\,\cdot\,))$$
in $D([0,1], \mathbb{R}^{2})$ with the weak $M_{2}$ topology.
\end{lem}
\begin{proof}
Let $C_{n} := (A_{n}, B_{n})$.
For $n=0,1,2,\ldots$ define functions $y_{n} \colon [0,1] \to \mathbb{R}^{2}$ by $y_{n}(t)=(x_{n}(t), 0)$.
Then clearly $y_{n} \in D([0,1], \mathbb{R}^{2})$. Since $x_{0}$ is continuous, by Corollary 12.11.5 in Whitt~\cite{Whitt02} and the definition of the metric $d_{p}$ in (\ref{e:defdp}) it follows that the function $h \colon D([0,1], \mathbb{R}^{2}) \to D([0,1], \mathbb{R}^{2})$ defined by $h(x) = x + y_{0}$ is continuous with respect to the weak $M_{2}$ topology.
Therefore by the continuous mapping theorem from (\ref{e:lemM2}) we obtain, as $n \to \infty$,
$ h(C_{n}) \dto h(C_{0})$, i.e.
\begin{equation}\label{e:lemM2a}
C_{n}(\,\cdot\,) + y_{0}(\,\cdot\,) \dto C_{0}(\,\cdot\,) + y_{0}(\,\cdot\,)
\end{equation}
in $D([0,1], \mathbb{R}^{2})$ under the weak $M_{2}$ topology.

If we show that
\begin{equation*}\label{e:lemM2b}
 \lim_{n \to \infty} \Pr [ d_{p}(C_{n}+y_{n}, C_{n}+y_{0}) > \delta ]=0
\end{equation*}
for any $\delta > 0$, then from (\ref{e:lemM2a}) by Slutsky's theorem (see Theorem 3.4 in Resnick~\cite{Resnick07}) we will have $C_{n}+y_{n} \dto C_{0} + y_{0}$ in $D([0,1], \mathbb{R}^{2})$ with the weak $M_{2}$ topology. Recalling the definition of the metric $d_{p}$ and the fact that the Skorohod $M_{2}$ metric on $D([0,1], \mathbb{R})$ is bounded above by the uniform metric on $D([0,1], \mathbb{R})$, we have
\begin{eqnarray*}
  \Pr [ d_{p}(C_{n}+y_{n}, C_{n}+y_{0}) > \delta ] &=& \Pr [ d_{M_{2}}(  x_{n}, x_{0} ) > \delta ]\\[0.6em]
  & \leq &   \Pr \bigg( \sup_{t \in [0,1]} \big| x_{n}(t) - x_{0}(t) \big| > \delta \bigg).
\end{eqnarray*}
Since $x_{n}(t) \to x_{0}(t)$ uniformly in $t$, we immediately obtain $\Pr [ d_{p}(C_{n}+y_{n}, C_{n}+y_{0}) > \delta ] \to 0$ as $n \to \infty$, and hence
$C_{n}+y_{n} \dto C_{0} + y_{0}$ as $n \to \infty$, i.e.
$$ (A_{n}(\,\cdot\,) + x_{n}(\,\cdot\,), B_{n}(\,\cdot\,)) \dto (A_{0}(\,\cdot\,) + x_{0}(\,\cdot\,), B_{0}(\,\cdot\,))$$
in $D([0,1], \mathbb{R}^{2})$
 with the weak $M_{2}$ topology.
\end{proof}

\section{Functional limit theorem}
\label{S:MA}

Let $(Z_{i})_{i \in \mathbb{Z}}$ be an i.i.d.~sequence of regularly varying random variables with index $\alpha \in (0,2)$. When $\alpha=1$, assume further that $Z_{1}$ is symmetric.
Let $\{\varphi_{i}, i=0,1,2,\ldots\}$ be a sequence of real numbers satisfying
\be\label{eq:InfiniteMAcond}
0 \le \sum_{i=0}^{s}\varphi_{i} \Bigg/ \sum_{i=0}^{\infty}\varphi_{i} \le 1, \qquad \textrm{for every} \ s=0, 1, 2 \ldots,
\ee
and such that the series defining the moving average process
$$ X_{i} = \sum_{j=0}^{\infty}\varphi_{j}Z_{i-j}, \qquad i \in \mathbb{Z}$$
is a.s.~convergent. We assume also that $\sum_{i=0}^{\infty}|\varphi_{i}| < \infty$. Hence $ \beta = \sum_{i=0}^{\infty}\varphi_{i}$ is finite.
Without loss of generality we may assume $\beta > 0$ (the case $\beta<0$ is treated analogously and is therefore omitted).
Let
$$\varphi_{+}=\max \{ \varphi_{j} \vee 0 : j \geq 0\}, \qquad \varphi_{-}= \max \{ -\varphi_{j} \vee 0 : j \geq 0\}.$$
Define further the corresponding partial sum and maxima processes
\be\label{eq:defVn}
V_{n}(t) = \frac{1}{a_{n}} \Bigg( \sum_{i=1}^{\floor {nt}}X_{i} - \floor {nt}b_{n}\Bigg), \qquad W_{n}(t)= \frac{1}{a_{n}} \bigvee_{i=1}^{\floor {nt}}X_{i}, \qquad t \in [0,1],
\ee
where the normalizing sequence $(a_n)$ satisfies~\eqref{eq:onedimregvar} and
$$ b_{n} = \left\{ \begin{array}{cc}
                                   0, & \quad \alpha \in (0,1], \\
                                   \beta \mathrm{E}(Z_{1}), & \quad \alpha \in (1,2).
                                 \end{array}\right.$$


\begin{thm}\label{t:FinMA}
Let $(Z_{i})_{i \in \mathbb{Z}}$ be an i.i.d.~sequence of regularly varying random variables with index $\alpha \in (0,2)$. When $\alpha=1$, suppose further that $Z_{1}$ is symmetric. Let $\{\varphi_{i}, i=0,1,2,\ldots\}$ be a sequence of real numbers satisfying (\ref{eq:InfiniteMAcond}), $\sum_{j=0}^{\infty}|\varphi_{j}| < \infty$ and $\varphi_{+}p + \varphi_{-}r>0$, with $p$ and $r$ as in (\ref{eq:pq}).
Then, as $n \to \infty$,
$$ L_{n}(\,\cdot\,) := (V_{n}(\,\cdot\,), W_{n}(\,\cdot\,)) \dto (\beta V(\,\cdot\,), W(\,\cdot\,))$$
in $D([0,1], \mathbb{R}^{2})$ endowed with the weak $M_{2}$ topology, where $V$ is an $\alpha$--stable L\'{e}vy process with characteristic triple $(0,\mu,b)$, with $\mu$ as in $(\ref{eq:mu})$ and
$$ b = \left\{ \begin{array}{cc}
                                   0, & \quad \alpha = 1\\[0.4em]
                                   (p-r)\frac{\alpha}{1-\alpha}, & \quad \alpha \in (0,1) \cup (1,2)
                                 \end{array}\right.,$$
and $W$ is an extremal process with exponent measure
$$\nu(dx)= (\varphi_{+}^{\alpha}p + \varphi_{-}^{\alpha}r) \alpha x^{-\alpha-1}1_{(0,\infty)}(x)\,dx.$$
\end{thm}


In the proof of the theorem we are going to use the following lemma.

\begin{lem}\label{t:FinMAlemma01}
Let
$$  V_{n}^{Z}(t) := \sum_{i=1}^{\floor {nt}}\frac{\beta Z_{i}-b_{n}}{a_{n}}, \quad  W_{n}^{Z}(t) := \bigvee_{i=1}^{\floor {nt}}\frac{ |Z_{i}|}{a_{n}}(\varphi_{+}1_{\{ Z_{i} >0 \}} + \varphi_{-} 1_{\{ Z_{i}<0 \}}), \qquad t \in [0,1].$$
Then, as $n \to \infty$,
\begin{equation}\label{e:pomkonv}
L_{n}^{Z}(\,\cdot\,) := (V_{n}^{Z}(\,\cdot\,), W_{n}^{Z}(\,\cdot\,)) \dto (\beta V(\,\cdot\,),  W(\,\cdot\,))
\end{equation}
 in $D([0,1], \mathbb{R}^{2})$ with the weak $M_{2}$ topology,
where $V$ is an $\alpha$--stable L\'{e}vy process with characteristic triple $(0,\mu,b)$ and $W$ is an extremal process
 with exponent measure $\nu(dx)= (\varphi_{+}^{\alpha}p + \varphi_{-}^{\alpha}r) \alpha x^{-\alpha-1}1_{(0,\infty)}(x)\,dx$.
\end{lem}
\begin{proof} ({\it Lemma~\ref{t:FinMAlemma01}})
Fix $0 < u < \infty$ and define the sum-maximum functional
$$ \Phi^{(u)} \colon \mathbf{M}_{p}([0,1] \times \EE) \to D([0,1], \mathbb{R}^{2})$$
by
$$ \Phi^{(u)} \Big(\sum_{i}\delta_{(t_{i}, x_{i})} \Big) (t)
  =  \Big( \beta \sum_{t_{i} \leq t}x_{i}\,1_{\{u < |x_{i}| < \infty \}},  \bigvee_{t_{i} \leq t} |x_{i}|(\varphi_{+} 1_{\{x_{i}>0\}} + \varphi_{-} 1_{\{x_{i}<0\}}) \Big)$$
  for $t \in [0,1]$
(here we for convenience set $\sup \emptyset = 0$), where the space $\mathbf{M}_p([0,1] \times \EE)$ of Radon point
measures on $[0,1] \times \EE$ is equipped with the vague
topology. Let $\mathbb{E}_{u} = \mathbb{E} \setminus [-u,u]$ and $\Lambda = \Lambda_{1} \cap \Lambda_{2}$, where
\begin{eqnarray*}
  \Lambda_{1} &=&  \{ \eta \in \mathbf{M}_{p}([0,1] \times \EE) :
   \eta ( \{0,1 \} \times \EE) = 0 = \eta ([0,1] \times \{ \pm \infty, \pm u \}) \} \\[0.6em]
  \Lambda_{2} &=& \{ \eta \in \mathbf{M}_{p}([0,1] \times \EE) :
  \eta ( \{ t \} \times \mathbb{E}_{u} ) \leq 1 \
  \text{for all $t \in [0,1]$} \}.
\end{eqnarray*}
The elements
of $\Lambda_2$ have no two atoms in $[0,1] \times \mathbb{E}_{u}$ with the same time
coordinate.

The functional $\Phi^{(u)}$ is continuous on the set $\Lambda$,
when $D([0,1], \mathbb{R}^{2})$ is endowed with the weak $M_{2}$ topology. Indeed, take an arbitrary $\eta \in \Lambda$ and suppose that $\eta_{n} \vto \eta$ in $\mathbf{M}_p([0,1] \times
\EE)$. We need to show that
$\Phi^{(u)}(\eta_n) \to \Phi^{(u)}(\eta)$ in $D([0,1],
\mathbb{R}^{2})$ according to the $WM_2$ topology. By
Theorem~12.5.2 in Whitt~\cite{Whitt02}, it suffices to prove that,
as $n \to \infty$,
$$ d_{p}(\Phi^{(u)}(\eta_{n}), \Phi^{(u)}(\eta)) =
\max_{k=1, 2}d_{M_{2}}(\Phi^{(u)}_{k}(\eta_{n}),
\Phi^{(u)}_{k}(\eta)) \to 0.$$
Following, with small modifications, the arguments in the proof of Lemma~3.2 in Basrak et al.~\cite{BKS} we obtain
$d_{M_{2}}(\Phi^{(u)}_{1}(\eta_{n}), \Phi^{(u)}_{1}(\eta)) \to 0$ as
$n \to \infty$.
Let
$$T= \{ t \in [0,1] : \eta (\{t\} \times \EE) = 0 \}.$$
Since $\eta$ is a Radon point measure, the set $T$ is dense in $[0,1]$. Fix $t \in T$ and take $\epsilon >0$ such that $\eta([0,t] \times \{\pm \epsilon\})=0$.
Later, when $\epsilon \downarrow 0$, we assume convergence to $0$ is through a sequence of values $(\epsilon_{j})$ such that $\eta([0,t] \times \{\pm \epsilon_{j}\})=0$ for all $j \in \mathbb{N}$ (this can be arranged since $\eta$ is a Radon point measure). Since the set $[0,t] \times \overline{\EE}_{\epsilon}$ is relatively compact in
$[0,1] \times \EE$, there exists a nonnegative integer
$k=k(\eta)$ such that
$$ \eta ([0,t] \times \overline{\EE}_{\epsilon}) = k < \infty.$$
By assumption, $\eta$ does not have any atoms on the border of the
set $[0,t] \times \overline{\EE}_{\epsilon}$. Therefore, by Lemma 7.1 in Resnick~\cite{Resnick07}, there exists a positive integer $n_{0}$ such
that
$$ \eta_{n} ([0,t] \times \overline{\EE}_{\epsilon})=k \qquad \textrm{for all} \ n \geq n_{0}.$$
Let
$(t_{i},x_{i})$ for $i=1,\ldots,k$ be the atoms of $\eta$ in
$[0,t] \times \overline{\EE}_{\epsilon}$. By the same lemma, the $k$ atoms
$(t_{i}^{(n)}, x_{i}^{(n)})$ of $\eta_{n}$ in $[0,t] \times \overline{\EE}_{\epsilon}$ (for $n \geq n_{0}$) can be labeled in such a way that
for every $i \in \{1,\ldots,k\}$ we have
$$ (t_{i}^{(n)}, x_{i}^{(n)}) \to (t_{i},x_{i}) \qquad \textrm{as}
\ n \to \infty.$$ In particular, for any $\delta >0$ we can find a
positive integer $n_{\delta} \geq n_{0}$ such that for all $n \geq
n_{\delta}$,
\begin{equation*}\label{e:etaconv}
 |t_{i}^{(n)} - t_{i}| < \delta \quad \textrm{and} \quad
 |x_{i}^{(n)}- x_{i}| < \delta \qquad \textrm{for} \ i=1,\ldots,k.
\end{equation*}
If $k=0$, then (for large $n$) the atoms of $\eta$ and $\eta_{n}$ in $[0,t] \times \EE$ are all situated in $[0,t] \times (-\epsilon, \epsilon)$. Hence
$ \Phi^{(u)}_{2}(\eta)(t) \in [0,\epsilon)$ and $ \Phi^{(u)}_{2}(\eta_{n})(t) \in [0, \epsilon)$, which imply
\begin{equation}\label{e:conv1}
  |\Phi^{(u)}_{2}(\eta_{n})(t) - \Phi^{(u)}_{2}(\eta)(t)| < \epsilon.
\end{equation}
If $k \geq 1$, take $\delta = \epsilon$. Note that $|x_{i}^{(n)}-x_{i}| < \delta$ implies $x_{i}^{(n)} >0$ iff $x_{i} >0$. Hence we have
\begin{eqnarray}\label{e:conv2}
  \nonumber |\Phi^{(u)}_{2}(\eta_{n})(t) - \Phi^{(u)}_{2}(\eta)(t)| &&\\[0.5em]
   \nonumber &\hspace*{-18em} =& \hspace*{-9em} \bigg| \bigvee_{i=1}^{k}|x_{i}^{(n)}|(\varphi_{+}1_{\{ x_{i}^{(n)}>0\}} + \varphi_{-}1_{\{ x_{i}^{(n)}<0\}}) - \bigvee_{i=1}^{k}|x_{i}|(\varphi_{+}1_{\{ x_{i}>0\}} + \varphi_{-}1_{\{ x_{i}<0\}}) \bigg|\\[0.5em]
   \nonumber & \hspace*{-18em} \leq &\hspace*{-9em} \bigvee_{i=1}^{k} \Big| (|x_{i}^{(n)}|-|x_{i}|) (\varphi_{+}1_{\{ x_{i}>0\}} + \varphi_{-}1_{\{ x_{i}<0\}})\Big| \leq  (\varphi_{+} \vee \varphi_{-}) \bigvee_{i=1}^{k}
    |x_{i}^{(n)}-x_{i}|\\[0.5em]
    & \hspace*{-18em} \leq &\hspace*{-9em} (\varphi_{+} \vee \varphi_{-})\,\epsilon,
\end{eqnarray}
    where the first inequality above follows from the following inequality
 \begin{equation*}\label{e:maxineq}
 \Big| \bigvee_{i=1}^{k}a_{i} - \bigvee_{i=1}^{k}b_{i} \Big| \leq
\bigvee_{i=1}^{k}|a_{i}-b_{i}|,
 \end{equation*}
which holds for arbitrary real numbers $a_{1}, \ldots, a_{k}, b_{1},
\ldots, b_{k}$.
Therefore form (\ref{e:conv1}) and (\ref{e:conv2}) we obtain
 $$\lim_{n \to \infty}|\Phi^{(u)}_{2}(\eta_{n})(t) -
 \Phi^{(u)}_{2}(\eta)(t)|< (\varphi_{+} \vee \varphi_{-} \vee 1)\,\epsilon,$$
  and if we let $\epsilon \to 0$, it follows that
 $\Phi^{(u)}_{2}(\eta_{n})(t) \to \Phi^{(u)}_{2}(\eta)(t)$ as $n \to
 \infty$. Note that $\Phi^{(u)}_{2}(\eta)$ and $\Phi^{(u)}_{2}(\eta_{n})$ are nondecreasing functions. Since, by
  Corollary 12.5.1 in Whitt~\cite{Whitt02}, $M_{1}$ convergence for monotone functions is equivalent to pointwise convergence in a dense subset of points plus convergence at the endpoints, and $M_{1}$ convergence implies $M_{2}$ convergence, we conclude that $d_{M_{2}}(\Phi^{(u)}_{2}(\eta_{n}),
 \Phi^{(u)}_{2}(\eta)) \to 0$ as $n \to \infty$. Hence
$\Phi^{(u)}$ is continuous at $\eta$.

Since the random variables $Z_{i}$ are i.i.d.~and regularly varying, Corollary 6.1 in Resnick~\cite{Resnick07} yields
\begin{equation}\label{e:ppconv}
 N_{n} := \sum_{i=1}^{n}\delta_{(\frac{i}{n}, \frac{Z_{i}}{a_{n}})} \dto N := \sum_{i}\delta_{(t_{i},j_{i})}, \qquad \textrm{as} \ n \to \infty,
\end{equation}
in $\mathbf{M}_{p}([0,1] \times \EE)$, where the limiting point process $N$ is a Poisson process with intensity measure $\emph{Leb} \times \mu$. Since $P( N \in \Lambda)=1$ (see Resnick~\cite{Resnick07}, page 221) and the functional $\Phi^{(u)}$ is continuous on the set $\Lambda$, from (\ref{e:ppconv}) by an application of the continuous mapping theorem we obtain
$ \Phi^{(u)}(N_{n})(\,\cdot\,) \dto \Phi^{(u)}(N)(\,\cdot\,)$ as $n \to \infty$, i.e.
\begin{eqnarray}\label{e:conv11}
 \nonumber L_{n}^{(u)}(\,\cdot\,) & := & \Big( \beta \sum_{i=1}^{\floor{n\,\cdot}} \frac{Z_{i}}{a_{n}} 1_{ \big\{ \frac{|Z_{i}|}{a_{n}} > u \big\} }, \bigvee_{i=1}^{\floor{n\,\cdot}}\frac{|Z_{i}|}{a_{n}} \big(\varphi_{+} 1_{\{ Z_{i}>0 \}} + \varphi_{-} 1_{\{ Z_{i}<0 \}}\big) \Big)\\[0.6em]
  & \hspace*{-5em} \dto & \hspace*{-2.5em} L^{(u)}_{0}(\,\cdot\,) :=
 \Big( \beta \sum_{t_{i} \leq\,\cdot}j_{i}1_{\{|j_{i}| >u \}}, \bigvee_{t_{i} \leq\,\cdot}|j_{i}| \big(\varphi_{+} 1_{\{ j_{i}>0 \}} + \varphi_{-} 1_{\{ j_{i}<0 \}} \big) \Big)
\end{eqnarray}
in $D([0,1], \mathbb{R}^{2})$ under the weak $M_{2}$ topology.
From (\ref{eq:onedimregvar}) we have, as $n \to \infty$,
\begin{eqnarray}\label{e:conv12}
 \nonumber \floor{nt} \mathrm{E} \Big( \frac{Z_{1}}{a_{n}} 1_{ \big\{ u < \frac{|Z_{1}|}{a_{n}} \leq 1 \big\} } \Big) &=& \frac{\floor{nt}}{n} \int_{u < |x| \leq 1}x n \Pr \Big( \frac{Z_{1}}{a_{n}} \in dx \Big) \\[0.6em]
   & \to &  t \int_{u < |x| \leq 1} x\mu(dx)
\end{eqnarray}
for every $t \in [0,1]$, and this convergence is uniform in $t$.
From (\ref{e:conv11}) and (\ref{e:conv12}), applying lemma~\ref{l:weakM2transf}, we obtain, as $n \to \infty$,
\begin{equation}\label{e:conv14}
\widetilde{L}_{n}^{(u)}(\,\cdot\,) \dto  L_{0}^{(u)}(\,\cdot\,) - x^{(u)}(\,\cdot\,)
\end{equation}
 in $D([0,1], \mathbb{R}^{2})$
 with the weak $M_{2}$ topology, where
$$ \widetilde{L}_{n}^{(u)}(t) = \Big( \beta \sum_{i=1}^{\floor{nt}} \frac{Z_{i}}{a_{n}} 1_{ \big\{ \frac{|Z_{i}|}{a_{n}} > u \big\} } - \beta \floor{nt} \mathrm{E} \Big( \frac{Z_{1}}{a_{n}} 1_{ \big\{ u < \frac{|Z_{1}|}{a_{n}} \leq 1 \big\} } \Big), \bigvee_{i=1}^{\floor{nt}}\frac{|Z_{i}|}{a_{n}} \big(\varphi_{+} 1_{\{ Z_{i}>0 \}} + \varphi_{-} 1_{\{ Z_{i}<0 \}}\big) \Big)$$
for $t \in [0,1]$,
and
$$x^{(u)}(t)= (ta_{u}, 0), \qquad a_{u} = \beta \int_{u < |x| \leq 1} x\mu(dx).$$

From the It\^{o} representation of a L\'{e}vy process (see Section 5.5.3 in Resnick~\cite{Resnick07} or Theorem 19.2 in Sato~\cite{Sa99}), there exists a L\'{e}vy process $V_{0}(\,\cdot\,)$ with characteristic triple $(0, \mu, 0)$ such that, as $u \to 0$,
$$ \sup_{t \in [0,1]} |L_{0\,1}^{(u)}(t) - ta_{u} - \beta V_{0}(t)| \stas 0.$$
Since uniform convergence implies (weak) $M_{2}$ convergence, it immediately follows that
$$d_{p}(L_{0}^{(u)}(\,\cdot\,) - x^{(u)}(\,\cdot\,), L(\,\cdot\,)) \to 0$$
almost surely as $u \to 0$, where
$$L(t) := \Big( \beta V_{0}(t), \bigvee_{t_{i} \leq t}|j_{i}| \big( \varphi_{+} 1_{\{ j_{i}>0 \}} + \varphi_{-} 1_{\{ j_{i}<0 \}} \big) \Big), \qquad t \in [0,1].$$
From this, since almost sure convergence implies convergence in distribution, we obtain, as $u \to 0$,
\begin{equation}\label{e:conv15}
 L_{0}^{(u)}(\,\cdot\,) - x^{(u)}(\,\cdot\,) \dto  L(\,\cdot\,)
\end{equation}
in $D([0,1], \mathbb{R}^{2})$ with the weak $M_{2}$ topology. Since $\sum_{i}\delta_{(t_{i},j_{i})}$ is a Poisson process with intensity measure $\emph{Leb} \times \mu$, the process
$$ W(t) :=  \bigvee_{t_{i} \leq t}|j_{i}| \big( \varphi_{+} 1_{\{ j_{i}>0 \}} + \varphi_{-} 1_{\{ j_{i}<0 \}} \big), \qquad t \in [0,1],$$
is an extremal process with exponent measure $\nu(dx)= (\varphi_{+}^{\alpha}p + \varphi_{-}^{\alpha}r) \alpha x^{-\alpha-1}1_{(0,\infty)}(x)\,dx$ (see Resnick~\cite{Re87}, Section 4.5, and Resnick~\cite{Resnick07}, page 161).

Let
$$  \widetilde{L}_{n}^{Z}(t) := \bigg( \sum_{i=1}^{\floor {nt}}\frac{\beta Z_{i}}{a_{n}} - \beta \floor{nt} \mathrm{E} \Big( \frac{Z_{1}}{a_{n}} 1_{ \big\{ \frac{|Z_{1}|}{a_{n}} \leq 1 \big\} } \Big), \bigvee_{i=1}^{\floor {nt}}\frac{ |Z_{i}|}{a_{n}} \big( \varphi_{+} 1_{\{ Z_{i}>0 \}} + \varphi_{-} 1_{\{ Z_{i}<0 \}} \big) \bigg)$$
for $t \in [0,1]$. If we show that
$$ \lim_{u \to 0} \limsup_{n \to \infty} \Pr [ d_{p}(\widetilde{L}_{n}^{Z}, \widetilde{L}_{n}^{(u)}) > \delta ]=0$$
for any $\delta > 0$, then from (\ref{e:conv14}), (\ref{e:conv15}) and a generalization of Slutsky's theorem (see Theorem 3.5 in Resnick~\cite{Resnick07}) we will have $\widetilde{L}_{n}^{Z} \dto  L$ as $n \to \infty$, in $D([0,1], \mathbb{R}^{2})$ with the weak $M_{2}$ topology.
Recalling the definitions and the fact that the metric $d_{M_{2}}$ is bounded above by the uniform metric, we have
\begin{eqnarray*}
 \Pr [ d_{p}(\widetilde{L}_{n}^{Z}, \widetilde{L}_{n}^{(u)}) > \delta ] &   & \\[0.6em]
   & \hspace*{-12em} \leq & \hspace*{-6em}  \Pr \bigg( \sup_{t \in [0,1]} \bigg| \sum_{i=1}^{\floor{nt}}\frac{\beta Z_{i}}{a_{n}} 1_{ \big\{ \frac{|Z_{i}|}{a_{n}} \leq u \big\}} - \beta \floor{nt} \mathrm{E} \Big( \frac{Z_{1}}{a_{n}} 1_{ \big\{ \frac{|Z_{i}|}{a_{n}} \leq u \big\}} \Big) \bigg|  > \delta \bigg)\\[0.6em]
& \hspace*{-12em} = & \hspace*{-6em}  \Pr \bigg( \max_{k=1,\ldots,n} \bigg| \sum_{i=1}^{k} \frac{ Z_{i}}{a_{n}} 1_{ \big\{ \frac{|Z_{i}|}{a_{n}} \leq u \big\}} - k \mathrm{E} \Big( \frac{Z_{1}}{a_{n}} 1_{ \big\{ \frac{|Z_{i}|}{a_{n}} \leq u \big\}} \Big) \bigg|  > \delta \beta^{-1} \bigg).
\end{eqnarray*}
In the i.i.d.~case it holds
$$ \lim_{u \to \infty} \limsup_{n \to \infty} \Pr \bigg( \max_{k=1,\ldots,n} \bigg| \sum_{i=1}^{k} \frac{ Z_{i}}{a_{n}} 1_{ \big\{ \frac{|Z_{i}|}{a_{n}} \leq u \big\}} - k \mathrm{E} \Big( \frac{Z_{1}}{a_{n}} 1_{ \big\{ \frac{|Z_{i}|}{a_{n}} \leq u \big\}} \Big) \bigg|  > \delta \beta^{-1} \bigg) = 0$$
(see the proof of Proposition 3.4 in Resnick~\cite{Re86}), and therefore, as $n \to \infty$,
\begin{equation}\label{e:conv16}
\widetilde{L}^{Z}_{n}(\,\cdot\,) \dto L(\,\cdot\,)
\end{equation}
in $D([0,1], \mathbb{R}^{2})$ with the weak $M_{2}$ topology.

Note that when $\alpha=1$ we have $\widetilde{L}^{Z}_{n} = L^{Z}_{n}$ (since $Z_{1}$ is symmetric) and the statement of the lemma holds. Therefore assume first $\alpha \in (0,1)$.
By Karamata's theorem, as $n \to \infty$,
$$ \floor{nt} \mathrm{E} \Big( \frac{Z_{1}}{a_{n}} 1_{ \big\{ \frac{|Z_{1}|}{a_{n}} \leq 1 \big\} } \Big) \to  t (p-r) \frac{\alpha}{1-\alpha}$$
for every $t \in [0,1]$. From this and (\ref{e:conv16}), applying Lemma~\ref{l:weakM2transf}, we obtain, as $n \to \infty$,
$$ \widetilde{L}^{Z}_{n}(\,\cdot\,) + \bigg( \beta \floor{n\,\cdot\,} \mathrm{E} \Big( \frac{Z_{1}}{a_{n}} 1_{ \big\{ \frac{|Z_{1}|}{a_{n}} \leq 1 \big\} } \Big), 0 \bigg) \dto L(\,\cdot\,) + \Big(  (\cdot) \beta (p-r) \frac{\alpha}{1-\alpha}, 0 \Big),$$
i.e.
\begin{equation}\label{e:conv17}
L_{n}^{Z}(\,\cdot\,) \dto \Big( \beta V_{0}(\,\cdot\,) + (\cdot) \beta (p-r) \frac{\alpha}{1-\alpha}, W(\,\cdot\,) \Big)
\end{equation}
in $D([0,1], \mathbb{R}^{2})$ with the weak $M_{2}$ topology. Put
$$V(t) := V_{0}(t) + t (p-r) \frac{\alpha}{1-\alpha}, \qquad t \in [0,1],$$
and note that (\ref{e:pomkonv}) holds in this case, since the characteristic triple of the L\'{e}vy process $V$ is $(0, \mu, (p-r) \alpha /(1-\alpha))$.

Finally assume $\alpha \in (1,2)$. By Karamata's theorem, as $n \to \infty$,
$$ \floor{nt} \mathrm{E} \Big( \frac{Z_{1}}{a_{n}} 1_{ \big\{ \frac{|Z_{1}|}{a_{n}} > 1 \big\}} \Big) \to t(p-r) \frac{\alpha}{\alpha -1},$$
for every $t \in [0,1]$. Therefore a new application of Lemma~\ref{l:weakM2transf} to (\ref{e:conv16})  yields, as $n \to \infty$,
$$ \widetilde{L}^{Z}_{n}(\,\cdot\,) - \bigg( \beta \floor{n\,\cdot\,} \mathrm{E} \Big( \frac{Z_{1}}{a_{n}} 1_{ \big\{ \frac{|Z_{1}|}{a_{n}} > 1 \big\} } \Big), 0 \bigg) \dto L(\,\cdot\,) - \Big(  (\cdot) \beta (p-r) \frac{\alpha}{\alpha-1}, 0 \Big),$$
i.e.
\begin{equation*}\label{e:conv20}
 L_{n}^{Z}(\,\cdot\,) \dto (\beta V(\,\cdot\,), W(\,\cdot\,))
\end{equation*}
in $D([0,1], \mathbb{R}^{2})$ with the weak $M_{2}$ topology, and this concludes the proof.
\end{proof}

\begin{rem}
From the proof of Lemma~\ref{t:FinMAlemma01} it follows that the components of the limiting process $(\beta V, W)$ can be expressed as functionals of the limiting point process $N = \sum_{i}\delta_{(t_{i},j_{i})}$ from relation (\ref{e:ppconv}), i.e.
$$ V(\,\cdot\,) = \lim_{u \to 0} \Big( \sum_{t_{i} \leq\,\cdot}j_{i}1_{\{|j_{i}| >u \}} - (\,\cdot\,) \int_{u < |x| \leq 1} x\mu(dx) \Big) + (\,\cdot\,) (p-r) \frac{\alpha}{1-\alpha} 1_{\{ \alpha \neq 0 \}},$$
where the limit holds almost surely uniformly on $[0,1]$, and
$$ W(\,\cdot\,) = \bigvee_{t_{i} \leq\,\cdot}|j_{i}| \big( \varphi_{+} 1_{\{ j_{i}>0 \}} + \varphi_{-} 1_{\{ j_{i}<0 \}} \big).$$
$N$ is a Poisson process with intensity measure $\emph{Leb} \times \mu$, and by using standard Poisson point process transformations (see proposition 5.2 and 5.3 in Resnick~\cite{Resnick07}) it can also be represented as
$$ N = \sum_{i}\delta_{(t_{i},P_{i}Q_{i})},$$
where
\begin{itemize}
  \item[(i)] $\sum_{i=1}^{\infty}\delta_{(t_{i}, P_{i})}$ is a Poisson point process on $[0,1] \times (0,\infty]$ with intensity measure $\emph{Leb} \times d(-x^{-\alpha})$;
  \item[(ii)] $(Q_{i})_{i \in \mathbb{N}}$ is a sequence of i.i.d.~random variables, independent of $\sum_{i=1}^{\infty}\delta_{(t_{i}, P_{i})}$, such that $\Pr(Q_{1}=1)=p$ and $\Pr(Q_{1}=-1)=r$.
\end{itemize}
\end{rem}

\begin{rem}
Lemma~\ref{t:FinMAlemma01} shows that the process $L_{n}^{Z}$  converges to $(\beta V, W)$ in the space $D([0,1], \mathbb{R}^{2})$ endowed with the weak $M_{2}$ topology. If we show that $L_{n}^{Z}$ is close to $L_{n}$ in a weak $M_{2}$ sense, then by the so called converging together result (i.e. Slutsky's theorem) it will follow that $L_{n}$ converges to the same limiting process. This is carried out in detail in the proof of Theorem~\ref{t:FinMA} below.

Heuristically, for a finite order moving average $X_{t}=\sum_{j=0}^{q}\varphi_{j}Z_{t-j}$, most of the sequence $Z_{i,n} :=Z_{i}/a_{n}$ is negligible, except for "big values" $Z_{i_{0},n}, Z_{i_{1},n}, \ldots, Z_{i_{k},n}, \ldots$ which are spread far apart. Note that a "big value" $Z_{i_{m},n}$ produces $q+1$ consecutive "big values" in the sequence $X_{t,n}=\sum_{j=0}^{q}\varphi_{j}Z_{t-j,n}$:
\begin{equation}\label{e:heur}
 X_{i_{m},n} \approx \varphi_{0}Z_{i_{m},n}, \quad X_{i_{m}+1,n} \approx \varphi_{1}Z_{i_{m},n}, \ldots, \ X_{i_{m}+q,n} \approx \varphi_{q}Z_{i_{m},n}.
\end{equation}
These values cover an interval on the $x$ axis of length $q/n$, and their sum is approximated well by $\sum_{j=0}^{q}\varphi_{j}Z_{i_{m},n} = \beta Z_{i_{m},n}$ when $n \to \infty$, showing that $V_{n}^{Z}$ is a suitable approximation of $V_{n}$.

As for the maxima process, a "big value" $\varphi_{j}Z_{i_{m},n}$ has an effect on $W_{n}$ only if it is positive, i.e. if $\varphi_{j}$ and $Z_{i_{m},n}$ are of the same sign. Hence the maximum of the values $X_{i_{m}+j,n}$ in (\ref{e:heur}) is approximated well by
\begin{eqnarray*}
  \bigvee_{j=0}^{q} \varphi_{j} Z_{i_{m},n} & = & \bigvee_{j=0}^{q} \varphi_{j}Z_{i_{m},n} \big( 1_{\{ \varphi_{j}>0,\,Z_{i_{m},n} >0 \}} + 1_{\{ \varphi_{j}<0,\,Z_{i_{m},n} <0 \}} \big)  \\[0.5em]
   &=& \bigvee_{j=0}^{q} |Z_{i_{m},n}| \big( \varphi_{j} 1_{\{ \varphi_{j}>0,\,Z_{i_{m},n} >0 \}} - \varphi_{j}  1_{\{ \varphi_{j}<0,\,Z_{i_{m},n} <0 \}} \big)\\[0.5em]
   & = & |Z_{i_{m},n}| \big( \varphi_{+} 1_{\{\,Z_{i_{m},n} >0 \}} + \varphi_{-} 1_{\{\,Z_{i_{m},n} <0 \}} \big)
\end{eqnarray*}
when $n \to \infty$, showing that $W_{n}^{Z}$ is an appropriate approximation of $W_{n}$.
\end{rem}
\smallskip

\begin{proof} ({\it Theorem~\ref{t:FinMA}})
We prove the theorem first for finite order moving average processes and then for infinite order moving averages. Hence, fix $q \in \mathbb{N}$ and let
$X_{i}=\sum_{j=0}^{q}\varphi_{j}Z_{i-j}$, $i \in \mathbb{Z}$. In this case condition (\ref{eq:InfiniteMAcond}) reduces to
\be\label{eq:FiniteMAcond}
0 \le \sum_{i=0}^{s}\varphi_{i} \Bigg/ \sum_{i=0}^{q}\varphi_{i} \le 1 \qquad \textrm{for every} \ s=0, 1, \ldots, q.
\ee
If we show  that for every $\delta >0$
$$ \lim_{n \to \infty} \Pr[d_{p}(L_{n}^{Z}, L_{n}) > \delta]=0,$$
then from Lemma~\ref{t:FinMAlemma01} by an application of Slutsky's theorem we will obtain $L_{n}(\,\cdot\,) \dto (\beta V(\,\cdot\,), W(\,\cdot\,))$ as $n \to \infty$,
in $D([0,1], \mathbb{R}^{2})$ endowed with the weak $M_{2}$ topology.
From the definition of the metric $d_{p}$ in (\ref{e:defdp}) it suffices to show that
\begin{equation}\label{e:BKsum}
\lim_{n \to \infty} \Pr[d_{M_{2}}(V_{n}^{Z}, V_{n}) > \delta]=0
\end{equation}
and
\begin{equation}\label{e:max1}
\lim_{n \to \infty} \Pr[d_{M_{2}}(W_{n}^{Z}, W_{n}) > \delta]=0.
\end{equation}
Relation (\ref{e:BKsum}) is established in the proof of Theorem 2.1 in Basrak and Krizmani\'{c}~\cite{BaKr14}. It remains to show (\ref{e:max1}).

Fix $\delta >0$ and let $n \in \mathbb{N}$ be large enough, i.e.
 $n > \max\{2q, 2q/\delta \}$.
Then by the definition of the metric $d_{M_{2}}$, we have
\begin{eqnarray*}
  d_{M_{2}}(W_{n}^{Z}, W_{n}) &=& \bigg(\sup_{v \in \Gamma_{W_{n}^{Z}}} \inf_{z \in \Gamma_{ W_{n}}} d(v,z) \bigg) \vee \bigg(\sup_{v \in \Gamma_{ W_{n}}} \inf_{z \in \Gamma_{W_{n}^{Z}}} d(v,z) \bigg) \\[0.4em]
   &= :& Y_{n} \vee T_{n}.
\end{eqnarray*}
Hence
\be\label{eq:AB}
\Pr [d_{M_{2}}(W_{n}^{Z}, W_{n})> \delta ] \leq \Pr(Y_{n}>\delta) + \Pr(T_{n}>\delta).
\ee
Now, we estimate the first term on the right hand side of (\ref{eq:AB}).
Let
$$ D_{n} = \{\exists\,v \in \Gamma_{W_{n}^{Z}} \ \textrm{such that} \ d(v,z) > \delta \ \textrm{for every} \ z \in \Gamma_{ W_{n}} \}.$$
Then by the definition of $Y_{n}$
\begin{equation}\label{e:Yn}
 \{Y_{n} > \delta\} \subseteq D_{n}.
\end{equation}
On the event $D_{n}$ it holds that $d(v, \Gamma_{ W_{n}})> \delta$. Let $v=(t_{v},x_{v})$. Then
\begin{equation}\label{e:i*}
 \Big| W_{n}^{Z} \Big( \frac{i^{*}}{n} \Big) - W_{n} \Big( \frac{i^{*}}{n} \Big) \Big| > \delta,
\end{equation}
where $i^{*}=\floor{nt_{v}}$ or $i^{*}=\floor{nt_{v}}-1$. Indeed, it holds that $t_{v} \in [i/n, (i+1)n)$ for some $i \in \{1,\ldots,n-1\}$ (or $t_{v}=1$). If $x_{v} = W_{n}^{Z}(i/n)$ (i.e. $v$ lies on a horizontal part of the completed graph), then clearly
$$\Big| W_{n}^{Z} \Big( \frac{i}{n} \Big) -  W_{n} \Big( \frac{i}{n} \Big) \Big| \geq d(v,  \Gamma_{ W_{n}}) > \delta,$$
and we put $i^{*}=i$.
 On the other hand, if $x_{v} \in [W_{n}^{Z}((i-1)/n), W_{n}^{Z}(i/n))$ (i.e. $v$ lies on a vertical part of the completed graph), one can similarly show that
 $$ \Big| W_{n}^{Z} \Big(\frac{i-1}{n} \Big) -  W_{n} \Big(\frac{i-1}{n} \Big) \Big| > \delta \qquad \textrm{if} \ W_{n} \Big( \frac{i^{*}}{n} \Big) > x_{v},$$
 and
 $$ \Big| W_{n}^{Z} \Big(\frac{i}{n} \Big) -  W_{n} \Big(\frac{i}{n} \Big) \Big| > \delta \qquad \textrm{if} \ W_{n} \Big( \frac{i^{*}}{n} \Big) < x_{v}.$$
In the first case put $i^{*}=i-1$ and in the second $i^{*}=i$.
Note that $i= \floor{nt_{v}}$, and therefore (\ref{e:i*}) holds.
Moreover, since $|i^{*}/n -(i^{*}+l)/n| \leq q/n \leq \delta$ for every $l=1,\ldots,q$ (such that $i^{*}+l \leq n$), from the definition of the set $D_{n}$ one can similarly conclude that
\begin{equation}\label{e:i*q}
 \Big| W_{n}^{Z} \Big( \frac{i^{*}}{n} \Big) - W_{n} \Big( \frac{i^{*}+l}{n} \Big) \Big| > \delta.
\end{equation}
Put $\gamma = \varphi_{+} \vee \varphi_{-}$.
We claim that
\begin{equation}\label{e:estim1}
 D_{n} \subseteq H_{n, 1} \cup H_{n, 2} \cup H_{n, 3} \cup H_{n, 4},
\end{equation}
where
\begin{eqnarray}
 \nonumber H_{n, 1} & = & \bigg\{ \exists\,l \in \{-q, \ldots, 0\} \ \textrm{such that} \ \frac{ |Z_{l}|}{a_{n}} > \frac{\delta}{4(q+1) \gamma} \bigg\},\\[0.4em]
 \nonumber H_{n, 2} & = & \bigg\{ \exists\,l \in \{1,\ldots,q\} \cup \{n-q+1, \ldots, n\} \ \textrm{such that} \ \frac{ |Z_{l}|}{a_{n}} > \frac{\delta}{4(q+1) \gamma} \bigg\},\\[0.4em]
  \nonumber H_{n, 3} & = & \bigg\{ \exists\,k \in \{1, \ldots, n\} \ \textrm{and} \ \exists\,l \in \{k-q,\ldots,k+q\} \setminus \{k\} \ \textrm{such that}\\[0.4em]
  \nonumber & & \ \frac{ |Z_{k}|}{a_{n}} > \frac{\delta}{4(q+1) \gamma} \ \textrm{and} \  \frac{ |Z_{l}|}{a_{n}} > \frac{\delta}{4(q+1) \gamma}  \bigg\},\\[0.4em]
  \nonumber H_{n, 4} & = & \bigg\{ \exists\,k \in \{1, \ldots, n\}, \ \exists\,j \in \{1,\ldots,n\} \setminus \{k,\ldots,k+q\}, \ \exists\,l_{1} \in \{0,\ldots,q\}\\[0.4em]
  \nonumber & & \textrm{and} \ \exists\,l \in \{0,\ldots,q\} \setminus \{l_{1}\} \ \textrm{such that} \ \frac{ |Z_{k}|}{a_{n}} > \frac{\delta}{4(q+1) \gamma},\\[0.4em]
  \nonumber & & \ \frac{ |Z_{j-l_{1}}|}{a_{n}} > \frac{\delta}{4(q+1) \gamma} \ \textrm{and} \  \frac{ |Z_{j-l}|}{a_{n}} > \frac{\delta}{4(q+1) \gamma}  \bigg\}.
\end{eqnarray}
To prove (\ref{e:estim1}) it suffices to show that
$$ D_{n} \cap (H_{n, 1} \cup H_{n, 2} \cup H_{n, 3})^{c} \subseteq H_{n, 4}.$$
Thus assume the event $D_{n} \cap (H_{n, 1} \cup H_{n, 2} \cup H_{n, 3})^{c}$ occurs. Then necessarily $W_{n}^{Z}(i^{*}/n) > \delta / [4(q+1)]$. Indeed, if $W_{n}^{Z}(i^{*}/n) \leq \delta / [4(q+1)]$, i.e.
$$ \bigvee_{j=1}^{i^{*}}\frac{ |Z_{j}|}{a_{n}} \big(\varphi_{+} 1_{\{ Z_{j}>0 \}} + \varphi_{-} 1_{\{ Z_{j}<0 \}} \big)   = W_{n}^{Z} \Big( \frac{i^{*}}{n} \Big) \leq \frac{\delta}{4(q+1)},$$
then for every $s \in \{ q+1, \ldots, i^{*}\}$ we have
\begin{eqnarray}\label{e:phipm}
  \nonumber \frac{X_{s}}{a_{n}} & \leq & \sum_{j=0}^{q} \frac{\varphi_{j} Z_{s-j}}{a_{n}} \leq  \sum_{j=0}^{q}  \frac{|Z_{s-j}|}{a_{n}} \big(\varphi_{+} 1_{\{ Z_{s-j}>0 \}} + \varphi_{-} 1_{\{ Z_{s-j}<0 \}} \big)\\[0.5em]
  & \leq & \frac{\delta}{4(q+1) }\,(q+1) = \frac{\delta}{4}.
\end{eqnarray}
Since the event $H_{n, 1}^{c} \cap H_{n, 2}^{c}$ occurs, for every $s \in \{1, \ldots, q\}$ we also have
\begin{eqnarray}\label{e:phipm2}
\nonumber  \frac{|X_{s}|}{a_{n}} & \leq & \sum_{j=0}^{q}|\varphi_{j}| \frac{|Z_{s-j}|}{a_{n}} \leq \frac{\delta}{4(q+1) \gamma} \sum_{j=0}^{q}|\varphi_{j}| \\[0.5em]
   & \leq & \frac{\delta}{4(q+1) \gamma}  \cdot (q+1) \gamma = \frac{ \delta}{4},
\end{eqnarray}
yielding
\begin{equation}\label{e:phipm3}
 -\frac{\delta}{4} \leq \frac{X_{1}}{a_{n}} \leq W_{n} \Big( \frac{i^{*}}{n} \Big) = \bigvee_{s=1}^{i^{*}}\frac{X_{s}}{a_{n}}  \leq \frac{\delta}{4}.
 \end{equation}
Hence
$$ \Big| W_{n}^{Z} \Big( \frac{i^{*}}{n} \Big) - W_{n} \Big( \frac{i^{*}}{n} \Big) \Big| \leq  \Big| W_{n}^{Z} \Big( \frac{i^{*}}{n} \Big) \Big| + \Big| W_{n} \Big( \frac{i^{*}}{n} \Big)\Big| \leq \frac{\delta}{4(q+1)} + \frac{\delta}{4} \leq \frac{\delta}{2},$$
which is in contradiction with (\ref{e:i*}).

Therefore $W_{n}^{Z}(i^{*}/n) > \delta / [4(q+1)]$.
This implies the existence of some $k \in \{1,\ldots,i^{*}\}$ such that
\begin{equation}\label{e:est31}
 W_{n}^{Z} \Big( \frac{i^{*}}{n} \Big) = \frac{ |Z_{k}|}{a_{n}} \big(\varphi_{+} 1_{\{ Z_{k}>0 \}} + \varphi_{-} 1_{\{ Z_{k}<0 \}} \big) > \frac{\delta}{4(q+1)}.
\end{equation}
Therefore
$$ \frac{|Z_{k}|}{a_{n}} \geq  \frac{|Z_{k}|}{a_{n}} \frac{\varphi_{+} 1_{\{ Z_{k}>0 \}} + \varphi_{-} 1_{\{ Z_{k}<0 \}} }{ \varphi_{+} \vee \varphi_{-} } > \frac{\delta}{4(q+1) \gamma},$$
and since $H_{n, 2}^{c}$ occurs, it follows that $q+1 \leq k \leq n-q$. Since $H_{n, 3}^{c}$ occurs, it holds that
\begin{equation}\label{e:est51}
\frac{ |Z_{l}|}{a_{n}}  \leq \frac{\delta}{4(q+1) \gamma} \qquad \textrm{for all} \ l \in \{k-q,\ldots,k+q\} \setminus \{k\}.
\end{equation}

Now we claim that $W_{n}(i^{*}/n) = X_{j}/a_{n}$ for some $j \in \{1,\ldots,i^{*}\} \setminus \{k,\ldots,k+q\}$. If this is not the case, then $W_{n}(i^{*}/n) = X_{j}/a_{n}$ for some $j \in \{k,\ldots,k+q\}$ (with $j \leq i^{*}$).
Here we distinguish two cases:
\begin{itemize}
\item[(i)] $k+q \leq i^{*}$. On the event $\{ Z_{k}>0 \}$ it holds that
 $$ |Z_{k}| \big(\varphi_{+} 1_{\{ Z_{k}>0 \}} + \varphi_{-} 1_{\{ Z_{k}<0 \}} \big) = \varphi_{+}Z_{k} = \varphi_{j_{0}}Z_{k}$$
 for some $j_{0} \in \{0,\ldots, q\}$ (with $\varphi_{j_{0}} \geq 0$).
Since $k+j_{0} \leq i^{*}$, we have
\begin{equation}\label{e:est6}
\frac{X_{j}}{a_{n}} = W_{n}  \Big( \frac{i^{*}}{n} \Big) \geq \frac{X_{k+j_{0}}}{a_{n}}.
\end{equation}
Taking into account the assumptions that hold in this case, we can write
$$ \frac{X_{j}}{a_{n}} = \frac{\varphi_{j-k}Z_{k}}{a_{n}} + \sum_{\scriptsize \begin{array}{c}
                          s=0  \\
                          s \neq j-k
                        \end{array}}^{q} \frac{\varphi_{s}Z_{j-s}}{a_{n}} =: \frac{\varphi_{j-k}Z_{k}}{a_{n}} + F_{1},$$
and
$$ \frac{X_{k+j_{0}}}{a_{n}} = \frac{\varphi_{j_{0}}Z_{k}}{a_{n}} + \sum_{\scriptsize \begin{array}{c}
                          s=0 \\
                          s \neq j_{0}
                        \end{array}}^{q}\frac{\varphi_{s}Z_{k+j_{0}-s}}{a_{n}} =: \frac{\varphi_{j_{0}}Z_{k}}{a_{n}} + F_{2}.$$
From relation (\ref{e:est51}) (similarly as in (\ref{e:phipm2})) we obtain
$$ |F_{1}| \leq  \frac{\delta}{4(q+1) \gamma} \cdot q \gamma < \frac{\delta}{4},$$
and similarly $|F_{2}| < \delta/4$. Since
$ \varphi_{j_{0}} - \varphi_{j-k} = \varphi_{+} - \varphi_{j-k} \geq 0$,
 from (\ref{e:est6}) it follows that
$$ 0 \leq  \frac{\varphi_{j_{0}} Z_{k} - \varphi_{j-k}Z_{k}}{a_{n}} \leq F_{1}-F_{2} \leq |F_{1}| + |F_{2}| < \frac{\delta}{2}.$$
By (\ref{e:i*}) we have
$$ \Big| \frac{\varphi_{j_{0}} Z_{k}}{a_{n}} - \frac{X_{j}}{a_{n}} \Big| = \Big| W_{n}^{Z} \Big( \frac{i^{*}}{n} \Big) - W_{n} \Big( \frac{i^{*}}{n} \Big) \Big| > \delta,$$
and hence
$$ \delta < \Big| \frac{\varphi_{j_{0}} Z_{k}}{a_{n}} -  \frac{\varphi_{j-k}Z_{k}}{a_{n}} - F_{1} \Big| \leq \Big| \frac{\varphi_{j_{0}} Z_{k}}{a_{n}} -  \frac{\varphi_{j-k}Z_{k}}{a_{n}} \Big| + |F_{1}| < \frac{\delta}{2} + \frac{\delta}{4} = \frac{3\delta}{4},$$
which is not possible. On the event $\{ Z_{k} <0 \}$ it holds that
 $ |Z_{k}| \big(\varphi_{+} 1_{\{ Z_{k}>0 \}} + \varphi_{-} 1_{\{ Z_{k}<0 \}} \big) = \varphi_{-}|Z_{k}| = \varphi_{i_{0}}Z_{k}$
 for some $i_{0} \in \{0,\ldots, q\}$ (with $\varphi_{i_{0}} \leq 0$). Repeating the arguments as before we similarly arrive at a contradiction. Therefore this case can not happen.
\item[(ii)] $k+q > i^{*}$. Note that in this case $k \leq j \leq i^{*} < k+q$. Let $s_{0} \in \{1,\ldots,q\}$ be such that $i^{*}+s_{0}=k+q$. Let
$$ W_{n} \Big( \frac{i^{*}+s_{0}}{n} \Big) = \frac{X_{p}}{a_{n}},$$
for some $p \leq k+q$. Since $W_{n}(i^{*}/n) \leq W_{n}((i^{*}+s_{0})/n)$, it holds that $j \leq p$. Then
$$ \frac{X_{p}}{a_{n}} = W_{n}  \Big( \frac{k+q}{n} \Big) \geq \frac{X_{k+j_{0}}}{a_{n}} \vee \frac{X_{k+i_{0}}}{a_{n}} $$
for $j_{0}$ and $i_{0}$ as in (i). By (\ref{e:i*q}) we have
$$ \Big| \frac{ |Z_{k}|}{a_{n}} \big(\varphi_{+} 1_{\{ Z_{k}>0 \}} + \varphi_{-} 1_{\{ Z_{k}<0 \}} \big)  - \frac{X_{p}}{a_{n}} \Big| = \Big| W_{n}^{Z} \Big( \frac{i^{*}}{n} \Big) - W_{n} \Big( \frac{i^{*}+s_{0}}{n} \Big) \Big| > \delta,$$
and repeating the arguments as in (i) (with $p$ instead of $j$ and $i^{*} + s_{0}$ instead of $i^{*}$) we conclude that this case also can not happen.
\end{itemize}
Hence indeed $W_{n}(i^{*}/n) = X_{j}/a_{n}$ for some $j \in \{1,\ldots,i^{*}\} \setminus \{k,\ldots,k+q\}$. Now we have three cases: A--all random variables $Z_{j-q}, \ldots, Z_{j}$ are "small", B--exactly one is "large" and C--at least two of them are "large" ($Z$ is "small" if $ |Z| /a_{n} \leq \delta/[4(q+1) \gamma]$, otherwise it is "large"). We will show that the first two cases are not possible.
\begin{itemize}
  \item[Case A:] $ |Z_{j-l}|/a_{n} \leq \delta/[4(q+1) \gamma]$ for every $l=0,\ldots,q$.
 This yields (as in (\ref{e:phipm2}))
$$ \Big| W_{n} \Big( \frac{i^{*}}{n} \Big) \Big| = \frac{|X_{j}|}{a_{n}} \leq  \frac{\delta}{4}.$$
Let $j_{0}$ and $i_{0}$ be as in (i) above (we take $j_{0}$ on the set $\{Z_{k}>0\}$ and $i_{0}$ on the set $\{Z_{k}<0\}$). If $k+q \leq i^{*}$, then
$$ \frac{X_{j}}{a_{n}} \geq \frac{X_{k+j_{0}}}{a_{n}} = \frac{\varphi_{j_{0}}Z_{k}}{a_{n}} + F_{2},$$
where $F_{2}$ is as in (i) above, with $|F_{2}| < \delta/4$. Therefore
$$  \frac{\varphi_{j_{0}}Z_{k}}{a_{n}} \leq \frac{X_{j}}{a_{n}} - F_{2} \leq \frac{|X_{j}|}{a_{n}} + |F_{2}| < \frac{\delta}{4} + \frac{\delta}{4} = \frac{\delta}{2},$$
and
$$ \Big| W_{n}^{Z} \Big( \frac{i^{*}}{n} \Big) - W_{n} \Big( \frac{i^{*}}{n} \Big) \Big| = \Big| \frac{\varphi_{j_{0}} Z_{k}}{a_{n}} - \frac{X_{j}}{a_{n}} \Big| \leq \frac{\varphi_{j_{0}} Z_{k}}{a_{n}} + \frac{|X_{j}|}{a_{n}} < \frac{\delta}{2} + \frac{\delta}{4} = \frac{3\delta}{4},$$
 which is in contradiction with (\ref{e:i*}). The same conclusion follows if $j_{0}$ is replaced by $i_{0}$ On the other hand, if $k+q > i^{*}$, let $s_{0}$ be as in (ii) above. Then, when $W_{n}((i^{*}+s_{0})/n)=X_{j}/a_{n}$, we similarly obtain a contradiction with (\ref{e:i*q}). Alternatively, when $W_{n}((i^{*}+s_{0})/n)=X_{p}/a_{n}$ for some $p \in \{i^{*}, \ldots, i^{*}+s_{0}\}$, in the same manner as in (ii) above we get a contradiction. Thus this case can not happen.
  \item[Case B:]  There exists $l_{1} \in \{0,\ldots,q\}$ such that $ |Z_{j-l_{1}}|/a_{n} > \delta/[4(q+1)\gamma]$ and  $ |Z_{j-l}|/a_{n} \leq \delta/[4(q+1)\gamma]$ for every $l \in \{0,\ldots,q\} \setminus \{l_{1}\}$. Assume first $k+q \leq i^{*}$. Here we analyze only what happens on the event $\{Z_{k}>0\}$ (the event $\{Z_{k}<0\}$ can be treated analogously and is therefore omitted). Then
\begin{equation}\label{e:est7}
\frac{X_{j}}{a_{n}} \geq \frac{X_{k+j_{0}}}{a_{n}} = \frac{\varphi_{j_{0}}Z_{k}}{a_{n}} + F_{2},
\end{equation}
where $j_{0}$ and $F_{2}$ are as in (i) above, with $|F_{2}| < \delta/4$.
 Write
$$ \frac{X_{j}}{a_{n}} = \frac{\varphi_{l_{1}}Z_{j-l_{1}}}{a_{n}} + \sum_{\scriptsize \begin{array}{c}
                          s=0  \\
                          s \neq l_{1}
                        \end{array}}^{q} \frac{\varphi_{s}Z_{j-s}}{a_{n}} =: \frac{\varphi_{l_{1}}Z_{j-l_{1}}}{a_{n}} + F_{3}.$$
Similarly as before we obtain $|F_{3}| < \delta/4$. Since
$$ W_{n}^{Z} \Big( \frac{i^{*}}{n} \Big) \geq \frac{|Z_{j-l_{1}}|}{a_{n}} \big(\varphi_{+} 1_{\{ Z_{j-l_{1}}>0 \}} + \varphi_{-} 1_{\{ Z_{j-l_{1}}<0 \}} \big) \geq \frac{\varphi_{l_{1}}Z_{j-l_{1}}}{a_{n}}$$
we have
$$ \frac{\varphi_{j_{0}} Z_{k}}{a_{n}} = \frac{|Z_{k}|}{a_{n}} \big(\varphi_{+} 1_{\{ Z_{k}>0 \}} + \varphi_{-} 1_{\{ Z_{k}<0 \}} \big) = W_{n}^{Z} \Big( \frac{i^{*}}{n} \Big) \geq \frac{\varphi_{l_{1}}Z_{j-l_{1}}}{a_{n}}, $$
which yields
\begin{equation}\label{e:est8}
\frac{ \varphi_{j_{0}} Z_{k}}{a_{n}} - \frac{X_{j}}{a_{n}} \geq \frac{\varphi_{l_{1}}Z_{j-l_{1}}}{a_{n}} - \frac{X_{j}}{a_{n}} = - F_{3}.
\end{equation}
Relations (\ref{e:est7}) and (\ref{e:est8}) yield
$$- (|F_{2}| + |F_{3}|) \leq -F_{3} \leq \frac{\varphi_{j_{0}} Z_{k}}{a_{n}} - \frac{X_{j}}{a_{n}} \leq -F_{2} \leq  |F_{2}| + |F_{3}|,$$
i.e.
$$ \Big| W_{n}^{Z} \Big( \frac{i^{*}}{n} \Big) - W_{n} \Big( \frac{i^{*}}{n} \Big) \Big| = \Big| \frac{\varphi_{j_{0}} Z_{k}}{a_{n}} - \frac{X_{j}}{a_{n}} \Big| \leq |F_{2}| + |F_{3}| < \frac{\delta}{4} + \frac{\delta}{4}=\frac{\delta}{2},$$
which is in contradiction with (\ref{e:i*}).
Alternatively assume $k+q > i^{*}$ and let $s_{0}$ be as in (ii) above. If $W_{n}((i^{*}+s_{0})/n)=X_{j}/a_{n}$, we similarly obtain a contradiction with (\ref{e:i*q}), and if $W_{n}((i^{*}+s_{0})/n)=X_{p}/a_{n}$ for some $p \in \{i^{*}, \ldots, i^{*}+s_{0}\}$, with the same reasoning as in (ii) we arrive at a contradiction. Hence this case also can not happen.
  \item[Case C:] There exist $l_{1} \in \{0,\ldots,q\}$ and $l \in \{0,\ldots,q\} \setminus \{l_{1}\}$ such that $ |Z_{j-l_{1}}|/a_{n} > \delta/[4(q+1)\gamma]$ and $ |Z_{j-l}|/a_{n} > \delta/[4(q+1)\gamma]$. In this case the event $H_{n, 4}$ occurs.
\end{itemize}
Therefore only Case C is possible, and this yields $D_{n} \cap (H_{n, 1} \cup H_{n, 2} \cup H_{n, 3})^{c} \subseteq H_{n, 4}$. Hence (\ref{e:estim1}) holds.
By stationarity we have
$$ \Pr(H_{n, 1}) \leq (q+1) \Pr \bigg( \frac{|Z_{1}|}{a_{n}} > \frac{\delta}{4(q+1) \gamma} \bigg), $$
and hence by the regular variation property we observe
\begin{equation}\label{e:est2}
\lim_{n \to \infty} \Pr(H_{n, 1})=0.
\end{equation}
Similarly
$$ \Pr(H_{n, 2}) \leq 2q \Pr \bigg( \frac{|Z_{1}|}{a_{n}} > \frac{\delta}{4(q+1) \gamma} \bigg), $$
and
\begin{equation}\label{e:est3}
\lim_{n \to \infty} \Pr(H_{n, 2})=0.
\end{equation}
Since $Z_{k}$ and $Z_{l}$ that appear in the formulation of $H_{n, 3}$ are independent, it follows that
$$ \Pr(H_{n, 3}) \leq \frac{2q}{n} \bigg[ n \Pr \bigg( \frac{|Z_{1}|}{a_{n}} > \frac{\delta}{4(q+1) \gamma} \bigg) \bigg]^{2}, $$
and hence
\begin{equation}\label{e:est5}
\lim_{n \to \infty} \Pr(H_{n, 3})=0.
\end{equation}
From the definition of the set $H_{n, 4}$ it follows that $k, j-l_{1}, j-l$ are all different, which implies that the random variables $Z_{k}$, $Z_{j-l_{1}}$ and $Z_{j-l}$ are independent. Using this and stationarity we obtain
$$ \Pr(H_{n, 4}) \leq \frac{q(q+1)}{n} \bigg[ n \Pr \bigg( \frac{|Z_{1}|}{a_{n}} > \frac{\delta}{4(q+1) \gamma} \bigg) \bigg]^{3}, $$
and hence we conclude
\begin{equation}\label{e:est10}
\lim_{n \to \infty} \Pr(H_{n, 4})=0.
\end{equation}
Now from (\ref{e:estim1}) and (\ref{e:est2})--(\ref{e:est10}) we obtain
$$ \lim_{n \to \infty} \Pr(D_{n})=0,$$
and hence (\ref{e:Yn}) yields
\begin{equation}\label{eq:Ynend}
\lim_{n \to \infty} \Pr(Y_{n}> \delta)=0.
\end{equation}

It remains to estimate the second term on the right hand side of (\ref{eq:AB}). Let
$$ E_{n} = \{\exists\,v \in \Gamma_{W_{n}} \ \textrm{such that} \ d(v,z) > \delta \ \textrm{for every} \ z \in \Gamma_{ W_{n}^{Z}} \}.$$
Then by the definition of $T_{n}$
\begin{equation}\label{e:Tnfirst}
 \{T_{n} > \delta\} \subseteq E_{n}.
\end{equation}
On the event $E_{n}$ it holds that $d(v, \Gamma_{ W_{n}^{Z}})> \delta$.
Interchanging the roles of the processes $W_{n}(\,\cdot\,)$ and $W_{n}^{Z}(\,\cdot\,)$, in the same way as before for the event $D_{n}$ it can be shown that
\begin{equation}\label{e:i*qneg}
 \Big| W_{n}^{Z} \Big( \frac{i^{*}-l}{n} \Big) - W_{n} \Big( \frac{i^{*}}{n} \Big) \Big| > \delta
\end{equation}
for all $l=0,\ldots, q$ (such that $i^{*}-l \geq 0$), where $i^{*}=\floor{nt_{v}}$ or $i^{*}=\floor{nt_{v}}-1$, and $v=(t_{v},x_{v})$.

Now we want to show that $E_{n} \cap (H_{n, 1} \cup H_{n, 2} \cup H_{n, 3})^{c} \subseteq H_{n, 4}$, and hence assume the event $E_{n} \cap (H_{n, 1} \cup H_{n, 2} \cup H_{n, 3})^{c}$ occurs. Since (\ref{e:i*qneg}) (for $l=0$) is in fact (\ref{e:i*}), repeating the arguments used for $D_{n}$ we conclude that (\ref{e:est31}) holds. Here we also claim that $W_{n}(i^{*}/n) = X_{j}/a_{n}$ for some $j \in \{1,\ldots,i^{*}\} \setminus \{k,\ldots,k+q\}$. Hence assume this is not the case, i.e. $W_{n}(i^{*}/n) = X_{j}/a_{n}$ for some $j \in \{k,\ldots,k+q\}$ (with $j \leq i^{*}$). We can repeat the arguments from (i) above to conclude that $k + q \leq i^{*}$ is not possible. It remains to see what happens when $k + q > i^{*}$. Let
$$ W_{n}^{Z} \Big( \frac{i^{*}-q}{n} \Big) = \frac{ |Z_{s}|}{a_{n}} \big(\varphi_{+} 1_{\{ Z_{s}>0 \}} + \varphi_{-} 1_{\{ Z_{s}<0 \}} \big)$$
for some $s \in \{1, \ldots, i^{*}-q\}$. Note that $i^{*}-q \geq 1$ since $q+1 \leq k  \leq i^{*}$. We distinguish two cases:
\begin{itemize}
\item[(a)] $W_{n}^{Z}(i^{*}/n) > W_{n}(i^{*}/n)$. In this case the definition of $i^{*}$ implies that $W_{n}(i^{*}/n) \leq x_{v} \leq W_{n}^{Z}(i^{*}/n)$. Since $|t_{v}-(i^{*}-q)/n| < (q+1)/n \leq \delta$,  from $d(v, \Gamma_{W_{n}^{Z}})> \delta$ we conclude
$$ \widetilde{d} \Big((x_{v}, \Big[ W_{n}^{Z} \Big( \frac{i^{*}-q}{n} \Big), W_{n}^{Z} \Big(\frac{i^{*}}{n} \Big) \Big] \Big) > \delta,$$
where $\widetilde{d}$ is the Euclidean metric on $\mathbb{R}$. This yields
$$W_{n}^{Z} \Big( \frac{i^{*}-q}{n} \Big) > W_{n} \Big( \frac{i^{*}}{n} \Big),$$
and from (\ref{e:i*qneg}) we obtain
\begin{equation}\label{e:estE1}
 W_{n}^{Z} \Big( \frac{i^{*}-q}{n} \Big) > W_{n} \Big( \frac{i^{*}}{n} \Big) + \delta.
\end{equation}
From this, taking into account relation (\ref{e:phipm3}), we obtain
$$ \frac{|Z_{s}|}{a_{n}} \geq   \frac{1}{\gamma}\,W_{n}^{Z} \Big( \frac{i^{*}-q}{n} \Big) > \frac{1}{\gamma} \Big( -\frac{\delta}{4} + \delta \Big) = \frac{3\delta}{4\gamma} > \frac{\delta}{4(q+1)\gamma},$$
and since $H_{n, 3}^{c}$ occurs it follows that
\begin{equation}\label{e:estE12}
\frac{|Z_{l}|}{a_{n}} \leq \frac{\delta}{4(q+1)\gamma} \quad \textrm{for every} \  l \in \{s-q,\ldots,s+q\} \setminus \{s\}.
\end{equation}
Let $p_{0} \in \{0,\ldots,q\}$ be such that $\varphi_{p_{0}}Z_{s} =  |Z_{s}| \big(\varphi_{+} 1_{\{ Z_{s}>0 \}} + \varphi_{-} 1_{\{ Z_{s}<0 \}} \big) $.
Since $s + p_{0} \leq i^{*}$, it holds that
\begin{equation}\label{e:estE2}
 \frac{X_{j}}{a_{n}} = W_{n} \Big( \frac{i^{*}}{n} \Big) \geq \frac{X_{s+p_{0}}}{a_{n}} = \frac{\varphi_{p_{0}}Z_{s}}{a_{n}} + F_{4},
\end{equation}
where
$$ F_{4} =  \sum_{\scriptsize \begin{array}{c}
                          m=0 \\
                          m \neq p_{0}
                        \end{array}}^{q}\frac{\varphi_{m}Z_{s+p_{0}-m}}{a_{n}}.$$
From (\ref{e:estE1}) and (\ref{e:estE2}) we obtain
$$ \frac{\varphi_{p_{0}} Z_{s}}{a_{n}} > \frac{X_{j}}{a_{n}} + \delta \geq \frac{\varphi_{p_{0}} Z_{s}}{a_{n}} + F_{4} + \delta,$$
i.e.
$\delta < - F_{4}$. But this is not possible since by (\ref{e:estE12})
$$ |F_{4}| \leq \frac{\delta}{4},$$
and we conclude that this case can not happen.
\item[(b)] $W_{n}^{Z}(i^{*}/n) \leq W_{n}(i^{*}/n)$. Then from (\ref{e:i*qneg}) we get
\begin{equation}\label{e:estE3}
 W_{n} \Big( \frac{i^{*}+s_{0}}{n} \Big) \geq W_{n} \Big( \frac{i^{*}}{n} \Big) \geq W_{n}^{Z} \Big( \frac{i^{*}}{n} \Big) + \delta,
\end{equation}
where $s_{0} \in \{1,\ldots,q \}$ is such that $i^{*}+s_{0}=k+q$. Hence
$$ \Big| W_{n}^{Z} \Big( \frac{i^{*}}{n} \Big) - W_{n} \Big(\frac{i^{*}+s_{0}}{n} \Big) \Big| > \delta,$$
and repeating the arguments from (ii) above we conclude that this case also can not happen.
\end{itemize}
Thus we have proved that  $W_{n}(i^{*}/n) = X_{j}/a_{n}$ for some $j \in \{1,\ldots,i^{*}\} \setminus \{k,\ldots,k+q\}$. Similar as before one can prove now that Cases A and B can not happen (when $k+q > i^{*}$ we use also the arguments from (a) and (b)), which means that only Case C is possible. In that case the event $H_{n,4}$ occurs, and thus we have proved that $E_{n} \cap (H_{n, 1} \cup H_{n, 2} \cup H_{n, 3})^{c} \subseteq H_{n, 4}$. Hence
$$ E_{n} \subseteq H_{n, 1} \cup H_{n, 2} \cup H_{n, 3} \cup H_{n, 4},$$
and from (\ref{e:est2})--(\ref{e:est10}) we obtain
$$ \lim_{n \to \infty} \Pr(E_{n})=0,$$
Therefore (\ref{e:Tnfirst}) yields
\be\label{eq:Tnend}
\lim_{n \to \infty} \Pr(T_{n}> \delta)=0.
\ee
Now from (\ref{eq:AB}), (\ref{eq:Ynend}) and (\ref{eq:Tnend}) we obtain (\ref{e:max1}),
and finally conclude that $L_{n}(\,\cdot\,) \dto (\beta V(\,\cdot\,), W(\,\cdot\,))$
in $D([0,1], \mathbb{R}^{2})$ with the weak $M_{2}$ topology.

Therefore we proved the theorem for finite order moving average processes. Using this we will obtain now the functional convergence of $L_{n}(\,\cdot\,)$ for infinite order moving averages. Let $X_{i}=\sum_{j=0}^{\infty}\varphi_{j}Z_{i-j}$, $i \in \mathbb{Z}$, and put
$$ \lambda = \left\{ \begin{array}{cl}
                                   \varphi_{+} \wedge \varphi_{-}, & \quad \textrm{if} \ \varphi_{+}>0 \ \textrm{and} \ \varphi_{-}>0, \\
                                   \varphi_{+}, & \quad \textrm{if} \ \varphi_{-}=0,\\
                                   \varphi_{-}, & \quad \textrm{if} \ \varphi_{+}=0.
                                 \end{array}\right.$$
 Since $\sum_{i=0}^{\infty}|\varphi_{i}| < \infty$, for large $q \in \mathbb{N}$ it holds that
$ \sum_{i=q}^{\infty}|\varphi_{i}| < \lambda$. Fix such $q$ and define
$$ X_{i}^{q} = \sum_{j=0}^{q-1}\varphi_{j}Z_{i-j} + \varphi'_{q} Z_{i-q} \qquad i \in \mathbb{Z},$$
where $\varphi'_{q}= \sum_{i=q}^{\infty}\varphi_{i}$,
and
$$ V_{n, q}(t) = \sum_{i=1}^{\floor{nt}} \frac{X_{i}^{q}-b_{n}}{a_{n}}, \qquad W_{n, q}(t) = \bigvee_{i=1}^{\floor{nt}} \frac{X_{i}^{q}}{a_{n}}, \qquad t \in [0,1],$$
where the sequence $(a_n)$ satisfies~\eqref{eq:onedimregvar} and
$$ b_{n} = \left\{ \begin{array}{cc}
                                   0, & \quad \textrm{if} \ \alpha \in (0,1], \\
                                   \beta \mathrm{E}(Z_{1}), & \quad \textrm{if} \ \alpha \in (1,2).
                                 \end{array}\right.$$
The coefficients $\varphi_{0}, \ldots, \varphi_{q-1}, \varphi'_{q}$ satisfy relation (\ref{eq:FiniteMAcond}) and from the definition of $\lambda$ it follows that
$$\max \{ \varphi_{j} \vee 0 : j =0,\ldots, q-1\} \vee (\varphi'_{q} \vee 0) = \varphi_{+}$$
 and
 $$ \max \{ -\varphi_{j} \vee 0 : j =0,\ldots, q-1\} \vee (-\varphi'_{q} \vee 0) = \varphi_{-}.$$
Therefore for the finite order moving average process $(X_{i}^{q})_{i}$ it holds that
$$ L_{n, q}(\,\cdot\,) := (V_{n, q}(\,\cdot\,), W_{n, q}(\,\cdot\,)) \dto (\beta V(\,\cdot\,), W(\,\cdot\,)) \qquad \textrm{as} \ n \to \infty,$$
in $D([0,1], \mathbb{R}^{2})$ with the weak $M_{2}$ topology. If we show that for every $\epsilon >0$
\begin{equation}\label{e:Slutskyinf01}
 \lim_{q \to \infty} \limsup_{n \to \infty}\Pr[d_{p}(L_{n, q}, L_{n})> \epsilon]=0,
\end{equation}
then by a generalization of Slutsky's theorem (see Theorem 3.5 in~\cite{Resnick07}) it will follow $L_{n}(\,\cdot\,) \dto (\beta V(\,\cdot\,), W(\,\cdot\,))$, as $n \to \infty$, in $D([0,1], \mathbb{R}^{2})$ with the weak $M_{2}$ topology. By the definition of the metric $d_{p}$ in (\ref{e:defdp}) and the fact that the metric $d_{M_{2}}$ on $D([0,1], \mathbb{R})$ is bounded above by the uniform metric on $D([0,1], \mathbb{R})$, it suffices to show that
$$ \lim_{q \to \infty} \limsup_{n \to \infty}\Pr \bigg( \sup_{0 \leq t \leq 1}|V_{n, q}(t) - V_{n}(t)|> \epsilon \bigg)=0$$
and
$$ \lim_{q \to \infty} \limsup_{n \to \infty}\Pr \bigg( \sup_{0 \leq t \leq 1}|W_{n, q}(t) - W_{n}(t)|> \epsilon \bigg)=0$$
Recalling the definitions, we have
\begin{equation*}
 \Pr \bigg( \sup_{0 \leq t \leq 1}|V_{n, q}(t) - V_{n}(t)|> \epsilon \bigg) \leq \Pr \bigg( \sum_{i=1}^{n}\frac{|X_{i}^{q}-X_{i}|}{a_{n}} > \epsilon \bigg)
\end{equation*}
and
\begin{equation*}
 \Pr \bigg( \sup_{0 \leq t \leq 1}|W_{n, q}(t) - W_{n}(t)|> \epsilon \bigg) \leq \Pr \bigg( \bigvee_{i=1}^{n}\frac{|X_{i}^{q}-X_{i}|}{a_{n}} > \epsilon \bigg) \leq \Pr \bigg( \sum_{i=1}^{n}\frac{|X_{i}^{q}-X_{i}|}{a_{n}} > \epsilon \bigg)
\end{equation*}
In the proof of Theorem 3.1 in~\cite{BaKr14} it has been shown that
\begin{equation*}\label{e:BKAT}
\lim_{q \to \infty} \limsup_{n \to \infty} \Pr \bigg( \sum_{i=1}^{n}\frac{|X_{i}^{q}-X_{i}|}{a_{n}} > \epsilon \bigg)=0.
\end{equation*}
Hence (\ref{e:Slutskyinf01}) holds, which means that $L_{n}(\,\cdot\,) \dto (\beta V(\,\cdot\,), W(\,\cdot\,))$, as $n \to \infty$, in $D([0,1], \mathbb{R}^{2})$ with the weak $M_{2}$ topology. This concludes the proof.
\end{proof}

\begin{rem}\label{r:jcM2M1}
Theorem~\ref{t:FinMA} gives functional convergence of the joint stochastic process $L_{n}(\,\cdot\,)$ in the space $D([0,1], \mathbb{R}^{2})$ endowed with the weak $M_{2}$ topology induced by the metric $d_{p}$ given in (\ref{e:defdp}). Since for the second coordinate of $L_{n}(\,\cdot\,)$, i.e.~the partial maxima process, functional convergence actually holds in the stronger $M_{1}$ topology (see for instance~\cite{BaTa16} and~\cite{Kr14}), one could raise a question whether it could be possible to obtain a sort of joint convergence of $L_{n}(\,\cdot\,)$ in the $M_{2}$ topology on the first coordinate and in the $M_{1}$ topology on the second coordinate. Precisely, does the functional convergence hold in the topology induced by the metric
$$ \widetilde{d_{p}}(x,y)= \max \{ d_{M_{2}}(x_{1},y_{1}), d_{M_{1}}(x_{2}, y_{2}) \}$$
 for $x=(x_{1}, x_{2}), y=(y_{1}, y_{2}) \in D([0,1],
 \mathbb{R}^{2})$?
 Here $d_{M_{1}}$ denotes the $M_{1}$ metric, defined by
$$ d_{M_{1}}(x_{1},x_{2})
  = \inf \{ \|r_{1}-r_{2}\|_{[0,1]} \vee \|u_{1}-u_{2}\|_{[0,1]} : (r_{i},u_{i}) \in \Pi(x_{i}), i=1,2 \}$$
for $x_{1},x_{2} \in D([0,1], \mathbb{R})$, where $\Pi(x)$ is the set of $M_{1}$ parametric representations of the completed graph $\Gamma_{x}$, i.e.~continuous nondecreasing functions $(r,u)$ mapping $[0,1]$ onto $\Gamma_{x}$.

If the space $D([0,1], \mathbb{R})$ with the $M_{2}$ topology is a Polish space (which to our best knowledge is still an open question, see~\cite{Bo18}, Remark 4.1), we could proceed similarly as in~\cite{Kr18} and the answer to the above question would be affirmative.

We will take another approach.
Repeating the arguments from the proof of Lemma~\ref{t:FinMAlemma01}, but with $d_{M_{1}}$ for the second components of the corresponding processes instead of $d_{M_{2}}$, we derive immediately that
$L_{n}^{Z}(\,\cdot\,) \dto (\beta V(\,\cdot\,),  W(\,\cdot\,))$
 in $D([0,1], \mathbb{R}^{2})$ with the topology induced by the metric $\widetilde{d_{p}}$. In order to obtain $L_{n}(\,\cdot\,) \dto (\beta V(\,\cdot\,),  W(\,\cdot\,))$ in the same topology, as in the proof of theorem~\ref{t:FinMA} it remains to show that
 \begin{equation*}\label{e:max111}
\lim_{n \to \infty} \Pr[d_{M_{1}}(W_{n}^{Z}, W_{n}) > \delta]=0
\end{equation*}
for all $\delta >0$ (compare this relation to (\ref{e:max1})). We will not pursue it here, since it would presumably require a lot of technical details connected to parametric representation machinery, but instead we will use relation (\ref{e:max1}) and the fact that the second coordinate of $L_{n}(\,\cdot\,)$ refers to nondecreasing functions. By Remark 12.8.1 in~\cite{Whitt02} the following metric is a complete metric topologically equivalent to $d_{M_{1}}$:
$$ {d_{M_{1}}^{*}}(x_{1}, x_{2}) = d_{M_{2}}(x_{1}, x_{2}) + \lambda (\widehat{\omega}(x_{1},\cdot), \widehat{\omega}(x_{2},\cdot)),$$
where $\lambda$ is the L\'{e}vy metric on a space of distributions
$$ \lambda (F_{1},F_{2}) = \inf \{ \epsilon >0 : F_{2}(x-\epsilon) - \epsilon \leq F_{1}(x) \leq F_{2}(x+\epsilon) + \epsilon \ \ \textrm{for all} \ x \},$$
and
$$ \widehat{\omega}(x,z) = \left\{ \begin{array}{cc}
                                   \omega(x,e^{z}), & \quad z<0,\\[0.4em]
                                   \omega(x,1), & \quad z \geq 0,
                                 \end{array}\right.$$
with
$$ \omega (x,\delta) = \sup_{0 \leq t \leq 1} \ \sup_{0 \vee (t-\delta) \leq t_{1} < t_{2} < t_{3} \leq (t+\delta) \wedge 1} \{\| x(t_{2}) - [x(t_{1}), x(t_{3})] \| \}$$
for $x \in D([0,1], \mathbb{R})$ and $\delta >0$.
Here $\|z-A\|$ denotes the distance between a point $z$ and a subset $A \subseteq \mathbb{R}$.

Since $W_{n}(\,\cdot\,)$ and $W_{n}^{Z}(\,\cdot\,)$ are nondecreasing, for $t_{1} < t_{2} < t_{3}$ it holds that
$ \|W_{n}(t_{2}) - [W_{n}(t_{1}), W_{n}(t_{3})] \|=0$, which yields $\omega(W_{n}, \delta)=0$ for all $\delta>0$, and similarly $\omega(W_{n}^{Z}, \delta)=0$. Hence $\lambda (W_{n}^{Z}, W_{n})=0$, and $d_{M_{1}}^{*}(W_{n}^{Z}, W_{n}) = d_{M_{2}}(W_{n}^{Z}, W_{n})$. Now from (\ref{e:max1}) we obtain
$$ \lim_{n \to \infty} \Pr[d_{M_{1}}^{*}(W_{n}^{Z}, W_{n}) > \delta]=0,$$
and conclude that $L_{n}(\,\cdot\,)$ converges in distribution to $(\beta V(\,\cdot\,),  W(\,\cdot\,))$ in the topology induced by the metric
$$ {d_{p}}^{*}(x,y)= \max \{ d_{M_{2}}(x_{1},y_{1}), d_{M_{1}}^{*}(x_{2}, y_{2}) \}$$
 for $x=(x_{1}, x_{2}), y=(y_{1}, y_{2}) \in D([0,1],
 \mathbb{R}^{2})$, i.e.~in the $M_{2}$ topology on the first coordinate of $L_{n}(\,\cdot\,)$ and in the $M_{1}$ topology on the second coordinate.
\end{rem}

\section*{Acknowledgements}
 The author would like to thank Hrvoje Planini\'{c} for valuable comments and suggestions which helped to improve the manuscript. This work has been supported in part by Croatian Science Foundation under the project 3526 and by University of Rijeka research grant 13.14.1.2.02.


\end{document}